\documentclass[english, reqno]{amsart}   
\usepackage[T1]{fontenc}
\usepackage[latin1]{inputenc}
\usepackage[english]{babel}
\usepackage{amsmath,amsfonts,amssymb,amsthm,verbatim, fancyvrb}
\usepackage[notcite,notref,final]{showkeys}
\usepackage{mathrsfs}
\usepackage{graphicx}
\usepackage{fixme}                    
\usepackage{color}

\newtheorem{theorem}{Theorem}[section]
\newtheorem{proposition}[theorem]{Proposition}
\newtheorem{lemma}[theorem]{Lemma}
\newtheorem{corollary}[theorem]{Corollary}
\theoremstyle{definition}
\newtheorem{definition}[theorem]{Definition}
\newtheorem{remark}[theorem]{Remark}
\newtheorem{example}[theorem]{Example}

\renewcommand{\phi}{\varphi}
\renewcommand{\epsilon}{\varepsilon}
\newcommand{\N}{\mathbb N}

\newcommand{\C}{\mathbb C}

\newcommand{\csa}{$C^*$-al\-ge\-bra}
\newcommand{\cssa}{$C^*$-sub\-al\-ge\-bra}
\newcommand{\shom}{$*$-homo\-mor\-phism}
\newcommand{\siso}{$*$-iso\-mor\-phism}
\newcommand{\su}{$\sigma$-unital}
\newcommand{\spu}{$\sigma_p$-unital}
\renewcommand{\subset}{\subseteq}

\renewcommand{\setminus}{\backslash}

\DeclareMathOperator{\BK}{\mathbb K} 
\newcommand{\A}{\mathfrak A} 
\newcommand{\B}{\mathfrak B} 
\newcommand{\I}{\mathfrak I} 
\newcommand{\E}{\mathcal E} 
\newcommand{\os}{\textnormal{os}} 
\newcommand{\col}{\textnormal c} 
\newcommand{\Mul}{\mathcal M} 
\newcommand{\Mat}{\mathsf M} 
\DeclareMathOperator{\Prim}{Prim} 

\newcommand{\CKsub}{\subset_{\textnormal{CK}}} 
\newcommand{\dm}{\succeq}

\newcommand{\gc}[1]{$#1$-coefficients}
\newcommand{\full}{full} 
\newcommand{\partitioned}{partitioned} 

\numberwithin{equation}{section}

\input xy
\xyoption{all}

\title{Hereditary $C^*$-subalgebras of graph $C^*$-algebras}
\author{Sara E. Arklint}
\address{Department of Mathematical Sciences, University of Copenhagen, Uni\-versi\-tets\-parken~5, DK-2100 Copenhagen, Denmark}
\email{arklint@math.ku.dk}

\author{James Gabe}
\address{Department of Mathematics and Computer Science, University of Southern Denmark, Campusvej~55, DK-5230  Odense, Denmark}
\curraddr{Mathematical Sciences, University of Southampton, University Road, Southampton, SO17 1BJ, United Kingdom}
\email{j.gabe@soton.ac.uk}

\author{Efren Ruiz}
\address{Department of Mathematics, University of Hawaii, Hilo, 200 W.~Kawili St., Hilo, Hawaii, 96720-4091 USA}
\email{ruize@hawaii.edu}

\date{\today}
\keywords{graph $C^*$-algebras, hereditary $C^*$-subalgebras}
\subjclass[2010]{Primary: 46L55}

\begin{document}

\begin{abstract}
We show that a \csa{} $\A$ which is stably isomorphic to a unital graph \csa{}, is isomorphic to a graph \csa{} if and only if it admits an approximate unit of projections.
As a consequence, a hereditary \cssa{} of a unital real rank zero graph \csa{} is isomorphic to a graph \csa{}.
Furthermore, if a \csa{} $\A$ admits an approximate unit of projections, then its minimal unitization is isomorphic to a graph \csa{} if and only if $\A$ is stably isomorphic to a unital graph \csa{}.
\end{abstract}

\bibliographystyle{alpha}
\maketitle

\section{Introduction}
Graph \csa{}s were introduced by M. Enomoto and Y. Watatani in 1980 as a generalization of the Cuntz--Krieger algebras.
For a graph \csa{} $C^*(E)$, many \csa{}ic properties correspond to properties of its underlying graph $E$, making graph \csa{}s \emph{\csa{}s we can see}.
Like Cuntz-Krieger algebras, graph \csa{}s provide models for purely infinite \csa{}s, both simple and with finitely many ideals.
In the simple case, any {stable} UCT Kirchberg algebra with free $K_1$-group is isomorphic to a graph \csa{}.
And in the non-simple case, the class of graph \csa{}s provide an abundance of concrete examples of \csa{}s with both purely infinite and stably finite parts.
For these reasons, graph \csa{}s show up in several contexts, including the classification program, as either counterexamples or test objects for conjectures and working theories.

In this paper, we study permanence properties for the class of graph \csa{}s.  An important observation for us is that all graph \csa{}s are \spu{}, i.e., they admit a countable approximate unit of projections.
Thus, having an approximate unit consisting of projections is a necessary condition for a \csa{} to be isomorphic to a graph \csa{}.  We show, in Theorem~\ref{t:main}, that it is also sufficient when the \csa{} is stably isomorphic to a unital graph \csa{}.  As a corollary, if $E$ is a graph with finitely many vertices (equivalently $C^*(E)$ is unital), then a hereditary \cssa{} of $C^*(E)\otimes\BK$ is isomorphic to a graph \csa{} if and only if it is \spu{}; see Corollary~\ref{c:hereditary}.  If, in addition, $E$ is furthermore assumed to satisfy Condition~(K) (equivalently the real rank of $C^*(E)$ is zero), then a \csa{} stably isomorphic to $C^*(E)$ is isomorphic to a graph \csa{}, and a hereditary \cssa{} of $C^*(E)\otimes\BK$ is isomorphic to a graph \csa{}.  It was proved by Crisp (see \cite[Theorem~3.5 and Lemma~3.6]{crisp}), that for a graph $E$ with finitely many vertices, if $X \subseteq E^0$ and $p_X$ is the sum of the vertex projections over $X$, then $p_X C^*(E) p_X$ is isomorphic to a graph \csa{}.  This result of Crisp becomes a special case of our result (see Corollary~\ref{c:hereditary}).  

These are surprising results, and as noted in Example~\ref{e:unital}, they do not hold in general for nonunital graph \csa{}s.
The results also hold within the class of Cuntz--Krieger algebras.  
As was shown by the first and last named authors in~\cite{ar-corners-ck-algs}, 
a unital \csa{} stably isomorphic to a Cuntz--Krieger algebra is isomorphic to a Cuntz--Krieger algebra, and a unital, hereditary \cssa{} of a Cuntz--Krieger algebra is isomorphic to a Cuntz--Krieger algebra.  Our methods and constructions generalize those in~\cite{ar-corners-ck-algs} to graphs with infinitely many edges.

As in~\cite{ar-corners-ck-algs}, the proofs are constructive.
When $\A$ is a \spu{} \csa{} stably isomorphic to a unital graph \csa{} $C^*(E)$, we describe how to build a graph $G$ with $C^*(G)\cong\A$, and how $G$ relates to $E$; see Theorem~\ref{t:main}.
For a general \spu{} \cssa{} of $C^*(E)\otimes\BK$, the underlying graph is also constructed; see Corollary~\ref{c:hereditary}.

Knowing how the graph $G$ relates to the graph $E$, allows one to investigate what properties of $C^*(E)$ are inherited by the \csa{} $\A\cong C^*(G)$.
In~\cite{elr}, our constructions are applied  to investigate a class of \csa{}s that turn out to be in the class of graph \csa{}s.
Furthermore, for graphs $G$ that are of the form we construct, the minimal unitization of $C^*(G)$ is easily seen to be isomorphic to a graph \csa{}.
This allows us to establish Theorem~\ref{thm:unitization} for \spu{} \csa{}s $\A$: The minimal unitization of $\A$ is isomorphic to a graph \csa{} if and only if $\A$ is stably isomorphic to a unital graph \csa{}.

The Leavitt path algebra associated to a graph $E$ is the universal algebra (over a fixed field) constructed from generators and relations which highly resemble the defining generators and relations of the graph $C^\ast$-algebra $C^\ast(E)$. It is an interesting problem to determine which results about graph $C^\ast$-algebras have analogues in the realm of Leavitt path algebras. We believe that the same methods and ideas, which are presented in this paper, should allow one to prove similar results for Leavitt path algebras. In particular, we believe that our strategy implies that any unital corner in a unital Leavitt path algebra is isomorphic to a Leavitt path algebra. One could argue, that this result is even more surprising than its $C^\ast$-analogue, as the lack of analytic structure makes the Leavitt path algebras even more rigid than their $C^\ast$-algebra cousins.

\section{Definitions and preliminaries}

Throughout the paper, unless stated otherwise, all graphs will be countable and directed.
All definitions in this section are standard definitions.

\begin{definition} \label{def:graph}
Let $E = (E^0,E^1,s_{E},r_{E})$ be a (countable, directed) graph. A Cuntz--Krieger $E$-family is a set of mutually orthogonal projections $\{ p_v \mid v \in E^0 \}$ and a set $\{ s_e \mid e \in E^1 \}$ of partial isometries satisfying the following conditions:
\begin{itemize}
	\item[(CK0)] $s_e^* s_f = 0$ if $e,f \in E^1$ and $e \neq f$,
	\item[(CK1)] $s_e^* s_e = p_{r_{E}(e)}$ for all $e \in E^1$,
	\item[(CK2)] $s_e s_e^* \leq p_{s_{E}(e)}$ for all $e \in E^1$, and,
	\item[(CK3)] $p_v = \sum_{e \in s_{E}^{-1}(v)} s_e s_e^*$ for all $v \in E^0$ with $0 < |s_{E}^{-1}(v)| < \infty$.
\end{itemize}
The \emph{graph \csa{}} $C^*(E)$ is defined as the universal $C^*$-algebra given by these generators and relations.
\end{definition}

\begin{definition}
Let $E$ be a graph.
A subgraph $G$ of $E$ is called a \emph{CK-subgraph}, written $G\CKsub E$, if $s_G^{-1}(v)=s_E^{-1}(v)$ for all $v\in G^0$.
\end{definition}

The notion of CK-subgraphs and more generally CK-morphisms between arbitrary directed graphs was introduced in~\cite{MR2508151} by Goodearl who showed for all fields $K$ that a CK-morphism $G\to E$ induces an injective $K$-algebra morphism $L_K(G)\to L_K(E)$ between the Leavitt path algebras.
By Lemma~2.8 of~\cite{ar-corners-ck-algs}, the same holds in the \csa{}ic setting, so in particular $C^*(G)$ embeds into $C^*(E)$ when $G\CKsub E$.
More concretely, if $\{q_v,t_e\mid v\in G^0, e\in G^1\}$ is a universal Cuntz-Krieger $G$-family generating $C^*(G)$ and  $\{p_v,s_e\mid v\in E^0, e\in E^1\}$ a universal Cuntz-Krieger $E$-family generating $C^*(E)$, then the assignment $q_v\mapsto p_v, t_e\mapsto s_e$ induces an injective \shom{} $C^*(G)\to C^*(E)$.

\begin{definition}
Let $E$ be a graph.
A \emph{loop} in $E$ is an edge $e$ for which $s_E(e)=r_E(e)$.
We say that a loop $e$ is \emph{based at $s_E(e)$}.
A \emph{path of length $n$} in $E$ is a sequence $e_1\cdots e_n$ of edges $e_i$ in $E$ with $s_E(e_i)=r_E(e_{i-1})$ for all $i\in\{2,\ldots,n\}$.
We consider vertices {as} paths of length 0, and let $E^*$ and $E^n$ denote the set of paths respectively paths of length $n$ in $E$.
We extend the range and source map to paths by $r_E(e_1\cdots e_n)=r_E(e_n)$ and $s_E(e_1\cdots e_n)=s_E(e_1)$, and by $s_E(v)=r_E(v)=v$.
A path $\alpha$ in $E$ of nonzero length with $s_E(\alpha)=r_E(\alpha)$ is called a \emph{cycle} and is considered based at $s_E(\alpha)$.
A cycle $\alpha=e_1\cdots e_n$ is called \emph{simple} if $r_E(e_i)\neq r_E(e_n)$ for all $i<n$.
\end{definition}

\begin{definition}
A graph $E$ satisfies \emph{Condition~(K)} if for all $v\in E^0$, there is either no cycle based in $v$ or two distinct {simple cycles} in $v$.
\end{definition}
By Theorem~2.5 of~\cite{hs:pi}, a graph \csa{} $C^*(E)$ has real rank zero if and only if its underlying graph $E$ satisfies Condition~(K).

\begin{definition}
For vertices $v$ and $w$ in $E$, we say that \emph{$v$ dominates $w$}, written $v\dm w$, if there is a path of nonzero length in $E$ from $v$ to $w$.  For $S \subseteq E^0$ and $v \in E^0$, we write $v \dm S$ if there exists $w \in S$ such that $v \dm w$.  

If there exists a path (possibly of length zero) from $v$ to $w$, then we write $v \geq w$.  If $v \geq w$ and $v \neq w$, then we write $v > w$.  
\end{definition}

Note that if $v \neq w$, then $v \dm w$ if and only if $v \geq w$ if and only if $v > w$.  In the case that $v = w$, $v \dm v$ establishes that there is a path of nonzero length from $v$ to $v$, where as $v \geq v$ does not.  Distinguishing these notions is crucial in the proofs of our results.

\begin{definition}
Let $E$ be a graph, and let $v$ be a vertex in $E$.
The vertex $v$ is called \emph{regular} if $s_E^{-1}(v)$ is finite and nonempty.
If $s_E^{-1}(v)$ is empty, $v$ is called a \emph{sink}, and if $s_E^{-1}(v)$ is infinite, $v$ is called an \emph{infinite emitter}.
If $r_E^{-1}(v)$ is empty, then $v$ is called a \emph{source}.
\end{definition}

\begin{definition}
Let $E$ be a graph.
A subset $H$ of $E^0$ is called \emph{hereditary} if for all $v\in H$ and $w\in E^0$, $v\geq w$ implies $w\in H$.
A subset $H$ of $E^0$ is called \emph{saturated} if for all regular vertices in $v\in E^0$, $r_E(s_E^{-1}(v))\subset H$ implies $v\in H$.
Given a saturated hereditary subset $H$ of $E^0$, a vertex $v\in E^0$ is called a \emph{breaking vertex for $H$} if $v$ is an infinite emitter and $s_E^{-1}(v)\cap r_E^{-1}(E^0\setminus H)$ is finite and nonempty.
\end{definition}

\begin{definition}
Let $E$ be a graph, and let $\{p_v,s_e\mid v\in E^0, e\in E^1\}$ denote a Cuntz--Krieger $E$-family generating $C^*(E)$.
An \emph{admissible pair $(H,S)$} in $E$ is a saturated hereditary subset $H$ of $E^0$ and a set $S$ of breaking vertices for $H$.
Let $(H,S)$ be an admissible pair.
We let $I_{(H,S)}$ denote the ideal in $C^*(E)$ generated by
\[ \{ p_v\mid v\in H \} \cup \left\{ p_v - \sum_{e\in s_E^{-1}(v)\cap r_E^{-1}(E^0\setminus H)} s_es_e^* \ \bigg| \ v\in S\right\}. \]
And we define a subgraph $E_{(H,S)}$ of $E$ by $E_{(H,S)}^0=H\cup S$
and $E_{(H,S)}^1=s_E^{-1}(H)\cup \{e\in s_E^{-1}(S) \mid r_E(e)\in H \}$.
\end{definition}

\begin{definition}
Let $E$ be a graph, and let $\{p_v,s_e\mid v\in E^0, e\in E^1\}$ denote a Cuntz--Krieger $E$-family generating $C^*(E)$.
Let $\gamma$ denote the gauge action on $C^*(E)$, i.e., the action $\gamma$ of the circle group $\mathbb T$ on $C^*(E)$ for which $\gamma_z(s_e)=zs_e$ and $\gamma_z(p_v)=p_v$ for all $z\in\mathbb T$, $e\in E^1$, and $v\in E^0$.
An ideal $\I$ in $C^*(E)$ is called \emph{gauge-invariant} if $\gamma_z(\I)\subseteq\I$ for all $z\in\mathbb T$.
\end{definition}

All ideals are assumed two-sided and closed.
By Theorem~3.6 of~\cite{bhrs}, all gauge-invariant ideals in a graph \csa{} $C^*(E)$ are of the form $I_{(H,S)}$ for $(H,S)$ an admissible pair in $E$.
Furthermore, the ideal $I_{(H,S)}$ is isomorphic to a full corner in $C^*(E_{(H,S)}) \otimes \BK$.

\begin{proposition}\label{prop: full corner gauge invariant}
Let $E$ be a graph and let $(H, S)$ be an admissible pair.  Then $I_{ ( H, S ) }$ and $I_{(H, S)} \otimes \BK$ are isomorphic to full corners of $C^* ( E_{ (H,S) } ) \otimes \BK$.
\end{proposition}

\begin{proof}
By Theorem~5.1 of \cite{rt:ideals}, $I_{(H,S)} \cong C^* ( \overline{E}_{ (H,S) } )$ (see Definition~4.1 of \cite{rt:ideals} for the definition of $\overline{E}_{(H,S)}$).  Since $\overline{E}_{(H,S)}$ is obtained from $E_{(H,S)}$ by adding regular sources to $E_{(H,S)}$, we have that $C^* ( E_{(H,S) } )$ is isomorphic to a full corner of $C^* ( \overline{E}_{ (H,S) } )$.  By Corollary~2.9 of~\cite{b:hereditary}, $C^* ( E_{(H,S)} ) \otimes \BK \cong C^* ( \overline{E}_{ (H,S) } ) \otimes \BK$.  Therefore, $I_{(H,S)} \cong C^* ( \overline{E}_{ (H,S) } )$ is isomorphic to a full corner of $C^* ( E_{ (H,S) } ) \otimes \BK$ and $I_{(H, S)} \otimes \BK$ is isomorphic to a full corner of $C^* ( E_{ (H,S) } ) \otimes \BK \otimes \BK \cong C^* ( E_{(H,S) } ) \otimes \BK$.
\end{proof}

\section{Move equivalence and stably complete graphs}

Move equivalence for graphs with finitely vertices was introduced by A.~S\o{}rensen in~\cite{as:geo}.
It was introduced to identify stable isomorphism between unital graph \csa{}s $C^*(E)$ and $C^*(F)$ on the level of the graphs $E$ and $F$.

\begin{definition}[Out-splitting Graph]\label{def:out-splitting}
Let $E$ be a graph and let $u \in E^{0}$.  Partition $s_{E}^{-1} (u)$ into disjoints nonempty sets $\E_{1}, \ldots, \E_n$.  Define the \emph{out-splitting graph} $E_{\os}$ as follows:
\begin{align*}
{E}_{\os}^{0} &= \left( E^{0} \setminus \{ u\} \right) \cup \{ u^{1} , \ldots, u^n \} \\
{E}_{\os}^{1} &= \left( E^{1} \setminus r_{E}^{-1} (u) \right) \cup \{ e^{1} , \ldots, e^{n} \mid e \in r_{E}^{-1} (u) \}
\end{align*}
For $e \notin r_{E}^{-1}(u)$, we let $r_{ E_{\os } } (e) = r_{E} (e)$ and for $e \in r_{E}^{-1}(u)$, we let  $r_{ E_{\os } } (e^{i}) = u^{i}$.  For $e \notin s_{E}^{-1} (u)$, we let $s_{E_{\os}} ( e) = s_{E} (e)$, for $e \in s_{E}^{-1} (u) \setminus r_{E}^{-1}(u)$, we let $s_{E_{\os}} ( e) = u^{i}$ if $e \in \E_{i}$, and for $e \in s_{E}^{-1} (u) \cap r_{E}^{-1}(u)$, we let $s_{E_{\os}} ( e^{j}) = u^{i}$ if $e \in \E_{i}$, for $i,j=1,\ldots, n$.
\end{definition}

The following proposition is a special case of Theorem~3.2 of \cite{tbdp:flow}.

\begin{proposition}[Move~(O)]\label{t:out-splitting-BP}
Let $E$ be a graph and let $u \in E^{0}$.  Partition $s_{E}^{-1} (u)$ into disjoint nonempty sets $\E_{1}, \ldots, \E_n$.  If at most one $\E_{i}$ is infinite, then there exists a \siso{} $\phi \colon  C^{*} (E) \rightarrow C^{*} ( E_{\os})$ such that 
\begin{align*}
\phi( p_{v} ) = 
\begin{cases}
q_{v}, &\text{if $v \neq u$} \\
q_{u^{1}} +\cdots + q_{u^{n}}, &\text{if $v = u$} 
\end{cases}
\end{align*}
and
\begin{align*}
\phi( s_{e} ) = 
\begin{cases}
t_{e}, &\text{if $e \notin r_{E}^{-1} (u) $} \\
t_{e^{1}} +\cdots + t_{e^{n}}, &\text{if $e \in r_{E}^{-1} (u)$},
\end{cases}
\end{align*}
where $\{ p_{v} , s_{e} \mid v \in E^{0} , e \in E^{1} \}$ and $\{ q_{v} , t_{e} \mid v \in E_{\os}^{0} , e \in E_{\os }^{1} \}$ are universal Cuntz--Krieger families generating $C^{*} (E)$ and $C^{*} ( E_{\os} )$ respectively.  {Consequently, $\phi \otimes \mathrm{id}_{\BK} \colon  C^{*} (E) \otimes \BK \rightarrow C^{*} ( E_{\os}) \otimes \BK$ is a \siso.}
\end{proposition}

In~\cite{as:geo}, A.~S\o{}rensen defines Move~(O) together with Move~(I) (in-split at a regular vertex), Move~(S) (removing a regular source), and Move~(R) (reduction).
The following definition is Definition~4.1 of~\cite{as:geo}.

\begin{definition}
Move equivalence $\sim_M$ is the smallest equivalence relation on graphs with finitely many vertices such that $E\sim_MG$ if $G$ differs, up to graph isomorphism, from $E$ by an application of one of the moves (O), (I), (S), or (R), {or their inverses}.
If $E\sim_MG$, we say that $E$ is \emph{move equivalent} to $G$.
\end{definition}

If one adds Move~(C), the Cuntz splice, to the list of moves, one obtains a weaker move equivalence $\sim_{C\negthinspace E}$.
In~\cite{errs}, it is shown that $E\sim_{C\negthinspace E}G$ if and only if $C^*(E)\otimes\BK\cong C^*(G) {\otimes \BK}$, provided $C^*(E)$ and $C^*(G)$ are unital {and of real rank zero.}
For our purposes, the following generalization of Proposition~\ref{t:out-splitting-BP} suffices.  We refer to~\cite{as:geo} for a proof.

\begin{theorem}\label{t:move-eq}
Let $E$ and $G$ be graphs with finitely many vertices.  If $E\sim_MG$, then $C^*(E)\otimes\BK\cong C^*(G)\otimes\BK$.
\end{theorem}

We will need the following moves that can be derived from move equivalence.  For collapsing (Proposition~\ref{t:collapsing}) and Move~(T) (Proposition~\ref{p:moveT}), we refer to Theorem~5.2 respectively Theorem~5.4 of~\cite{as:geo}.

\begin{definition}[Collapsing]\label{def:collapsing} 
Let $E$ be a graph and let $u \in E^{0}$.  Assume that $u$ does not support a loop.
Defining the \emph{collapsing graph} $E_{\col}$ by $E_{\col}^{0} = E^{0} \setminus \{ u \}$, 
\begin{align*}
E_{\col}^{1} = \left( E^{1} \setminus \left( s_{E}^{-1}(u) \cup r_{E}^{-1}(u) \right) \right) \bigcup \{ [ef]  \mid e \in r_{E}^{-1}(u) , f \in s_{E}^{-1} (u)\},
\end{align*}  
the range and source maps extends those of $E$, $r_{E_{\col}} ( [ ef ] ) = r_{E} (f)$ and $s_{E_{\col}} ( [ ef ] ) = s_{E} (e)$.
\end{definition}

\begin{proposition}\label{t:collapsing}
Let $E$ be a graph and let $u \in E^{0}$.  Suppose $u$ is a regular vertex that is not a source and 
does not support a loop,
and let $E_\col$ denote the graph obtained by collapsing $u$.
Then $E\sim_ME_\col$, and
there exists a \siso{} $\phi \colon  C^{*} (E) \otimes \BK \rightarrow C^{*} (E_{\col} ) \otimes \BK$ such that $\phi ( p_{v} \otimes e_{1,1} ) \sim q_{v} \otimes e_{1,1}$ for all $v \in E^{0} \setminus \{ u \}$,
where $\{ p_{v} , s_{e} \mid v \in E^{0} , e \in E^{1} \}$ and $\{ q_{v} , t_{e} \mid v \in E_{\col}^{0} , e \in E_{\col }^{1} \}$ are universal Cuntz--Krieger families generating $C^{*} (E)$ and $C^{*} ( E_{\col} )$ respectively and $\{ e_{i,j} \}_{ i,j \in \N}$ is a system of matrix units of $\BK$. 
\end{proposition}

\begin{proposition}[Move~(T)]\label{p:moveT}
Let $E$ be a graph and let $\alpha=e_1\cdots e_n$ be a path in~$E$.
Assume that $s_E^{-1}(s_E(e_1))\cap r_E^{-1}(r_E(e_1))$ is infinite.
Let $G$ be the graph defined by $G^0=E^0$,
\[ G^1 = E^1 \cup \{ \alpha^m \mid m\in\N \}, \]
and $r_G$ and $s_G$ extending $r_E$ and $s_E$ respectively with $r_G(\alpha^m)=r_E(\alpha)$ and $s_G(\alpha)=s_E(\alpha)$.
Then $E\sim_MG$.
\end{proposition}

The following lemma tells us that, up to move equivalence, we can remove breaking vertices from a graph with finitely many vertices.
Two edges $e$ and $f$ in a graph~$E$ are called \emph{parallel} if $s_E(e)=s_E(f)$ and $r_E(e)=r_E(f)$.

\begin{lemma} \label{remove-breaking} 
Let $E$ be a graph with finitely many vertices, and let $u \in E$ be an infinite emitter.  Put
\[
	\E_1 = \{ e \in s_{E}^{-1}(u) \mid \text{there are infinitely many edges parallel to } e \} 
\]
and $\E_2 = s_{E}^{-1}(u) \setminus \E_1.$
Let $F$ be the graph obtained by out-splitting the vertex $u$ into the vertices $u^1, u^2$ according to the partition $\E_1, \E_2$.  Then 
\begin{enumerate}
\item $u^{2}$ is a finite emitter;

\item $u^{1}$ has the property that if $e \in s_{F} ^{-1} ( u^{1} )$, then there are infinitely many edges parallel to $e$;

\item the number of infinite emitters in $E$ is equal to the number of infinite emitters in $F$; and 

\item $C^*(E)\cong C^*(F)$, and $E\sim_MF$.
\end{enumerate}
\end{lemma}

\begin{proof}
Note that if $\E_{2} = \emptyset$, then $E = F$.  Suppose $\E_{2} \neq \emptyset$.  Since $E$ has finitely many vertices, we have that $\E_{2}$ is finite.  Therefore, $E\sim_MF$ via Move~(O), hence $C^*(E)\cong C^*(F)$.  It is clear from the construction of the out-splitting graph that $u^{2}$ is a finite emitter, and that $u^{1}$ has the property that if $e \in s_{F} ^{-1} ( u^{1} )$, then there are infinitely many edges parallel to $e$.
\end{proof}

\begin{definition} \label{d:canonical}
A graph $E$ is \emph{stably complete} if
\begin{enumerate}
\item its set of vertices $E^0$ is finite, \label{cf:fin}
\item every regular vertex $v$ of $E$ supports a loop, \label{cf:reg}
\item every vertex $v$ of $E$ supporting 2 distinct simple cycles, supports 2 loops, \label{cf:loops}
\item every infinite emitter $v$ of $E$ emits infinitely to every vertex it dominates, i.e., if $v\dm w$, then $s_E^{-1}(v)\cap r_E^{-1}(w)$ is infinite, \label{cf:inf}
\item for all vertices $v$ and $w$ in $E$, if $v$ dominates $w$ then $v$ emits to $w$, i.e., if $v\dm w$ then $s_E^{-1}(v)\cap r_E^{-1}(w)\neq\emptyset$, and \label{cf:emit}
\item if $v$ is an infinite emitter in $E$ supporting a loop, then there exists a regular vertex $w$ in $E$ with $v\dm w$ and $w\dm v$. \label{cf:inf to finite}
\end{enumerate}
\end{definition}

\begin{proposition} \label{p:canonical}
Let $E$ be a graph with finitely many vertices. Then there exists a stably complete graph $G$ for which 
$E\sim_MG$.
\end{proposition}
\begin{proof}
We first apply Lemma~\ref{remove-breaking} to each infinite emitter in $E$ to get a graph $E_{1}$ such that $E\sim_ME_1$, $E_{1}$ has finitely many vertices, and for every infinite emitter $v \in E_{1}^{0}$, if $e \in s_{E_{1}}^{-1} (v)$, then there exists infinitely many edges parallel to $e$.

For each infinite emitter $v$ and each $w$ with $v\dm w$, we can now apply Proposition~\ref{p:moveT} since all paths out of $v$ trivially satisfy the required condition.
We thereby achieve a graph $E_2$ with finitely many vertices for which $E_1\sim_ME_2$ and \eqref{cf:inf} of Definition~\ref{d:canonical} holds.

Now remove all regular sources of $E_{1}$, remove the regular vertices that become regular sources, and repeat this procedure finitely many times, so that we get a subgraph $E_{3}$ of $E_{2}$ with no regular sources.
Since $E_3$ is achieved from $E_2$ via Move~(S), $E_3\sim_ME_2$.

By Proposition~\ref{t:collapsing}, we may collapse all regular vertices $v$ in $E_3$ that is not a base point of at least one loop, to get a graph $E_{4}$ satisfying \eqref{cf:fin}, \eqref{cf:reg}, and \eqref{cf:inf}, for which $E_3 \sim_M E_4$.  For each infinite emitter $v$ and for each $w \in E^{0}$, with $v \dm w$ choose $e_{w} \in s^{-1}(v) \cap r^{-1} (w)$.  Then partition $s^{-1}(v)$ using the partition $\mathcal{E}_{1} = \{ e_{w} : v \dm w \}$ and $\mathcal{E}_{2} = s^{-1} (v) \setminus \mathcal{E}_{1}$.  Applying Move~(O) to $v$ with respect to this partition and doing this for all infinite emitters, we get a graph $E_{5}$ such that $E_{4} \sim_{M} E_{5}$ such that $E_{5}$ satisfies \eqref{cf:fin}, \eqref{cf:reg}, \eqref{cf:inf}, \eqref{cf:inf to finite}, and if $v$ is an infinite emitter and $| s^{-1} (w) \cap r^{-1} (v) | \geq 1$, then $| s^{-1} (w) \cap r^{-1} (w') | \geq 1$ for all $v \dm w'$. 

Throughout the rest of the proof, $\mathsf{A}_{F}$ will denote the adjacency matrix of $F$, i.e., $\mathsf{A}_{F} (u,v) = | s^{-1} (u) \cap r^{-1} (v) |$.  Let $\mathsf{E}_{u,v}$ denote the matrix that is the identity except at the entry $(u,v)$ in which $\mathsf{E} ( u, v ) = 1$.  Note that $\mathsf{E}_{u,v}$ acts on the left by adding adding row $v$ to row $u$ and $\mathsf{E}_{u,v}$ acts on the right by adding column $u$ to column $v$.  We will also use the convention that $\infty - 1 = \infty$.  We now do ``legal column operations'' as in \cite[Proposition~3.2]{errs-Cuntz-splice} (see also \cite[Lemma~7.2]{as:geo}) to obtain a new graph $E_{6} \sim_{M} E_{5}$, $E_{6}$ satisfies  \eqref{cf:fin}, \eqref{cf:reg}, \eqref{cf:loops}, \eqref{cf:inf}, \eqref{cf:emit}, and \eqref{cf:inf to finite}.  Note that if $v$ is an infinite emitter in $E_{5}$ such that $v \dm w$, then $v$ emits to $w$ by \eqref{cf:inf}.  Suppose $v$ is a regular vertex and $\mu = \mu_{1} \cdots \mu_{n}$ is a path through distinct vertices with source $v$ and $n \geq 2$.  Let $E_{6}$ be the graph with adjacency matrix satisfying
\[
\mathsf{A}_{E_{6}} - \mathsf{I} = ( \mathsf{A}_{E_{5}} - \mathsf{I} )E_{ s( \mu_{2} ), r( \mu_{2} ) } E_{ s( \mu_{3} ), r( \mu_{3} ) } \cdots E_{ s( \mu_{n} ), r( \mu_{n} ) }  
\]   
Since $\mathsf{A}_{E_{5}} ( s( \mu_{i} ) , r( \mu_{i} ) ) > 0$ for all $i$, by \cite[Proposition~3.2]{errs-Cuntz-splice} (see also \cite[Lemma~7.2]{as:geo}), $E_{5} \sim_{M} E_{6}$.  Note that $\mathsf{A}_{E_{6}} (v, r( \mu_{n} ) )  \geq \mathsf{A}_{E_{5}} ( v, r( \mu_{n} ) ) + 1$.  Moreover, if $r(\mu) = v$, then $\mathsf{A}_{E_{6}}( v, v ) \geq 2$.  So, in particular, if $v$ is a regular vertex that supports 2 distinct simple cycles, then $\mathsf{A}_{E_{6}}( v, v ) \geq 2$ since for $\mathsf{A}_{E_{5}} (v, v ) = 1$ and $v$ is a regular vertex that supports 2 distinct paths, there must be a path of length greater than or equal to 2 from $v$ to $v$.

Continuing this process finitely many times for all regular vertices $v$ and all $v \dm w$ with $s^{-1} (v) \cap r^{-1} (w) = \emptyset$, we get $E_{7} \sim_{M} E_{6}$ and $E_{7}$ satisfies \eqref{cf:fin}, \eqref{cf:reg}, \eqref{cf:loops}, \eqref{cf:inf}, \eqref{cf:emit}, and \eqref{cf:inf to finite}.  
\end{proof}

\section{Projections in unital graph $C^*$-algebras}
For a stable $C^\ast$-algebra $\A$ let $s_1,s_2,\dots \in \Mul(\A)$ be a sequence of isometries such that $\sum_{k=1}^\infty s_k s_k^\ast$ converges strictly to $1_{\Mul(\A)}$.
For a projection $p$ in $\A$ and $n\in \N \cup \{ \infty \}$ we let $n p$ denote the projection
\[
 \sum_{k=1}^n s_k p s_k^\ast.
\]
For $n= \infty$ this sum converges strictly to a multiplier projection. 
This construction is unique up to unitary equivalence, in the sense that if $t_1,t_2,\dots$ is another such sequence of isometries, then $\sum_{k=1}^\infty t_k s_k^\ast$ converges strictly to a unitary $u$ such that
\[
 u^\ast \left( \sum_{k=1}^n t_k p t_k^\ast \right) u = \sum_{k=1}^n s_k p s_k^\ast
\]
for all $n\in \N \cup \{\infty\}$. Thus we will in general not specify our chosen sequence of isometries.
It is obvious, that if $p\in \A$ is a projection then $p$ is Murray--von Neumann equivalent (denoted by $\sim$) to $1p$.

Also, if $I$ is some countable index set, and $(p_i)_{i\in I}$ is a collection of projections, then we let
\[
\bigoplus_{i\in I,f} p_i = \sum_{k=1}^{|I|} s_k p_{f(k)} s_k^\ast,
\]
where $f \colon \{1,\dots, |I|\} \to I$ is a bijection. Here $\{ 1,\dots, |I|\} = \mathbb N$ if $|I| = \infty$. As above, this does not depend on choice of bijection $f$ up to unitary equivalence, and thus we will omit $f$ from the notation above. We also define the sum $p \oplus q = s_1 p s_1^\ast + s_2 q s_2^\ast$. This sum is commutative and associative up to unitary equivalence, by similar arguments as above, so we allow ourselves to write $p_1 \oplus p_2 \oplus p_3$ (and similar expressions) without emphasising if this means $(p_1 \oplus p_2) \oplus p_3$ or $p_1 \oplus (p_2 \oplus p_3)$.

We will always consider a $C^\ast$-algebra $\A$ as a $C^\ast$-subalgebra of $\A \otimes \mathbb K$, via the isomorphism $id_\A \otimes e_{1,1} \colon \A \xrightarrow{\cong} \A \otimes e_{1,1} \subseteq \A \otimes \mathbb K$. So, for instance, if we are given projections $p,q\in \A$, then the sum $p \oplus q$ makes sense in $\A \otimes \mathbb K$.

\begin{definition}[\gc{E}]
Let $E$ be a graph and let $\{ p_{v} , s_{e} \mid v \in E^{0} , E^{1} \}$ be a universal Cuntz--Krieger family generating $C^{*}(E)$. Given a projection $p$ in $C^*(E)\otimes\BK$, a set of positive integers $\{n_{(v,T)}\}_{(v,T)\in S}$ indexed by a finite subset $S$ of $E^0\times 2^{E^1}$ is called \emph{\gc{E} for $p$} if
\begin{enumerate}
\item for all $(v,T)\in S$, $T$ is a finite subset of $s_E^{-1}(v)$; \label{gcone}
\item for all $(v,T)\in S$, $T\neq\emptyset$ only if $s_E^{-1}(v)$ is infinite; \label{gctwo}
\item the projection 
\[
\bigoplus_{(v,T)\in S}n_{(v,T)}\left(p_v-\sum_{e\in T} s_es_e^*\right)
\]
(which exists due to \eqref{gcone}--\eqref{gctwo})
is Murray--von Neumann equivalent to $p$ in $C^*(E)\otimes\BK$.
\end{enumerate}
\end{definition}

\begin{theorem}[cf. {\cite[Theorem~3.4 and Corollary~3.5]{HLMRT-non-stable-K-thy} and \cite[Theorem~3.5]{amp:nonstablekthy}}]\label{t:non-stable-kthy} 
Let $E$ be a graph and let $\{ p_{v} , s_{e} \mid v \in E^{0} , E^{1} \}$ be a universal Cuntz--Krieger family generating $C^{*}(E)$. 
Then any projection $p$ in $C^*(E)\otimes\BK$ has \gc{E}.
\end{theorem}

\begin{corollary} \label{c:non-stable-kthy}
Let $E$ be a graph.
If $\I$ is an ideal in $C^*(E)$ generated by projections, then $\I$ is gauge-invariant.
In particular, there exists an admissible pair $(H,V)$ in $E^0$ for which $\I=I_{(H,V)}$.
\end{corollary}
\begin{proof}
Let $\{p_v,s_e\mid v\in E^0, e\in E^1\}$ be a universal Cuntz-Kriger $E$-family generating $C^*(E)$, and let 
 $\{e_{i,j}\}_{i,j}$ denote a system of matrix units for $\BK$.
Let $\mathcal{P}$ be a set of projections in $C^*(E)$ such that the ideal generated by $\mathcal{P}$ is $\I$, and let \gc E $\{n_{(v,T)}\}_{(v,T)\in S_p}$ for $p\otimes e_{1,1}$ be given for $p \in \mathcal{P}$.
Then $\I$ is generated as an ideal by
\[ \bigcup_{p \in \mathcal{P} } \left\{ p_v - \sum_{e\in T}s_es_e^* \;\middle|\; (v,T)\in S_p \right\}. \]
As these generators are fixed under the gauge action, $\I$ is gauge-invariant.
\end{proof}

\begin{definition}\label{def:part}
Let $E$ be a graph and let $\{ p_{v} , s_{e} \mid v \in E^{0} , e \in E^{1} \}$ be a universal Cuntz--Krieger family generating $C^{*}(E)$.  Let $\{q_k\}_{k=1}^\infty$ be a sequence of projections in $C^*(E)\otimes\BK$ and let \gc E $\{n^{(k)}_{(v,T)}\}_{(v,T)\in S_k}$ for each $q_k$ be given.
The sequence $\{\{n^{(k)}_{(v,T)}\}_{(v,T)\in S_k}\}_{k=1}^\infty$ of \gc E is called \emph{\partitioned} if 
\begin{enumerate}
\item for all $k,l\in\N$ and $(v,T)\in S_k$, $(u,V)\in S_l$, if $T\cap V\neq\emptyset$ then $k=l$, $v=u$ and $T=V$; \label{partone}
\item for all $v,w\in E^0$ with $s_E^{-1}(v)\cap r_E^{-1}(w)$ infinite, \label{parttwo}
\[ \left( s_E^{-1}(v)\cap r_E^{-1}(w)\right)\setminus \bigcup_{k=1}^\infty\bigcup_{(v,T)\in S_k} T \]
is infinite.
\end{enumerate}
\end{definition}

\begin{lemma} \label{l:fullproj}
Let $E$ be a stably complete graph and let $\{ p_{v} , s_{e} \mid v \in E^{0} , e \in E^{1} \}$ be a universal Cuntz--Krieger family generating $C^{*}(E)$.  Let $p$ be a full projection in $C^*(E)\otimes\BK$.  Then $p$ has \gc E $\{m_{(v,T)}\}_{(v,T)\in \mathbb{S}}$ satisfying that for all $v\in E^0$ there exists $T\subset s_E^{-1}(v)$ for which $(v,T)\in \mathbb{S}$.
\end{lemma}
\begin{proof}
By Theorem~\ref{t:non-stable-kthy}, $p$ has \gc E $\{n_{(v,T)}\}_{(v,T)\in S}$. Let $H$ denote the hereditary closure in $E^0$ of the set
\[ V_0 = \{ v\in E^0 \mid \exists T\subset s_E^{-1}(v)\colon (v,T)\in S \} . 
\]
The projection $\bigoplus_{(v,T)\in S} n_{(v,T)} (p_v - \sum_{e\in T}s_e s_e^\ast)$ is Murray--von Neumann equivalent to $p$ and thus full, and it is Murray--von Neumann subequivalent to $n(\sum_{v\in V_0} p_v)$, where $n = \sum_{(v,T)\in S}n_{(v,T)}$.
It clearly follows that $\sum_{v\in V_0} p_v$ is full, and thus the hereditary and saturated set generated by $V_0$ is all of $E_0$. 
Since $E$ is stably complete, all subsets of $E^0$ are saturated and thus $H = E^0$.

Define recursively
\[ V_i = r_E(s_E^{-1}(V_{i-1}))\cup V_{i-1} . \]
Since $H=E^0$, there exists a $j$ for which $V_j=E^0$.
We will recursively construct \gc E $\{n^i_{(v,T)}\}_{(v,T)\in S_i}$ for $p$ satisfying
\[ V_i \subset \{ v\in E^0 \mid \exists T\subset s_E^{-1}(v)\colon (v,T)\in S_i \} .\]
Then the \gc E $\{n^j_{(v,T)}\}_{(v,T)\in S_j}$ will have the desired property.

Assume that \gc E $\{n^{i-1}_{(v,T)}\}_{(v,T)\in S_{i-1}}$ have been constructed (with $S_0=S$ and $n^0_{(v,T)}=n_{(v,T)}$).
Let $v\in V_i\setminus V_{i-1}$, and let $w\in V_{i-1}$ and $f\in s_E^{-1}(w)$ be given with $r_E(f)=v$.
If the vertex $w$ is regular, there exists a loop $f'$ based in $w$, and
\begin{equation} p_w = \sum_{e\in s_E^{-1}(w)} s_es_e^* \sim p_w\oplus p_v\oplus \bigoplus_{e\in s_E^{-1}(w)\setminus\{f,f'\}}p_{r_E(e)} . \label{eq:wreg} \end{equation}
Replace $p_w$ in
\[ p \sim \bigoplus_{(u,V)\in S_{i-1}} n^{i-1}_{(u,V)}\left(p_u-\sum_{e\in V}s_es_e^*\right) \]
using \eqref{eq:wreg} using that $(w,\emptyset)\in S_{i-1}$ as $w\in V_{i-1}$.
If the vertex $w$ is not regular, it is an infinite emitter and there exists $T\subset s_E^{-1}(w)$ with $(w,T)\in S_{i-1}$.
Since $E$ is stably complete, we may assume $f\notin T$ by replacing $f$ if necessary.
Then 
\begin{equation} p_w - \sum_{e\in T}s_es_e^* = p_w + s_fs_f^* - \sum_{e\in T\cup\{f\}}s_es_e^* \sim \left(p_w-\sum_{e\in T\cup\{f\}}s_es_e^*\right) + p_v . \label{eq:wsing}\end{equation}
Replace $p_w-\sum_{e\in T}s_es_e^*$ in
\[ p \sim \bigoplus_{(u,V)\in S_{i-1}} n^{i-1}_{(u,V)}\left(p_u-\sum_{e\in V}s_es_e^*\right) \]
using \eqref{eq:wsing} using that $(w,T)\in S_{i-1}$.
Make such replacements recursively for all the finitely many $v$ in $V_i\setminus V_{i-1}$ to create \gc E $\{n^i_{(u,V)}\}_{(u,V)\in S_i}$.
\end{proof}

\begin{definition}
Let $E$ be a graph and let $\{ p_{v} , s_{e} \mid v \in E^{0} , e\in E^{1} \}$ be a universal Cuntz--Krieger family generating $C^{*}(E)$.  Let $\{q_k\}_{k=1}^\infty$ be a sequence of projections in $C^*(E)\otimes\BK$ and let \gc E $\{n^{(k)}_{(v,T)}\}_{(v,T)\in S_k}$ for each $q_k$ be given.
The sequence $\{\{n^{(k)}_{(v,T)}\}_{(v,T)\in S_k}\}_{k=1}^\infty$ of \gc E is called \emph{\full} if for all $v\in E^0$ there exists $T\subset s_E^{-1}(v)$ for which $(v,T)\in S_k$ for some $k\in\N$.
For a projection $q$ in $C^*(E)\otimes\BK$, its \gc E $\{n_{(v,T)}\}_{(v,T)\in S}$ are called \full{} if the constant sequence $\{\{n_{(v,T)}\}_{(v,T)\in S}\}_{k=1}^\infty$ is full.
\end{definition}

\begin{lemma} \label{l:partitioned}
Let $E$ be a graph with finitely many vertices and let $\{ p_{v} , s_{e} \mid v \in E^{0} , E^{1} \}$ be a universal Cuntz--Krieger family generating $C^{*}(E)$. 
Then any sequence of projections $\{q_k\}_{k=1}^\infty$ in $C^*(E)\otimes\BK$ has a \partitioned{} sequence of \gc{E}.
If a sequence of projections $\{q_k\}_{k=1}^\infty$ in $C^*(E)\otimes\BK$ admits a \full{} sequence of \gc{E}, then it admits a \full{} \partitioned{} sequence of \gc{E}.
\end{lemma}
\begin{proof}
By Theorem~\ref{t:non-stable-kthy}, each $q_k$ has \gc E $\{n^{(k)}_{(v,T)}\}_{(v,T)\in S_k}$.  Using that $p_v-s_es_e^*\sim p_v-s_fs_f^*$ holds when $s_E(f)=s_E(g)=v$ and $r_E(f)=r_E(g)$, one may achieve \eqref{partone}--\eqref{parttwo} of Definition~\ref{def:part} by replacing the $T\neq\emptyset$ suitably.
As the replacements do not affect the set of $v\in E^0$ for which there exists $T\subset s_E^{-1}(v)$ with $(v,T)\in S_k$, being \full{} is not affected.
\end{proof}

\begin{lemma} \label{l:infemit_loop}
Let $E$ be a stably complete graph and let $\{p_u, s_e \mid u\in E^0, e\in E^1\}$  denote a universal Cuntz--Krieger family generating $C^*(E)$.
Let $v$ be an infinite emitter in $E$ supporting a loop, and let $T$ be a finite subset of $s_E^{-1}(v)$.
Then there exists a family $\{n_u\}_{v\dm u}$ of nonnegative integers for which
\[ p_v-\sum_{e\in T}s_es_e^* \sim p_v\oplus\bigoplus_{v\dm u} n_up_u \]
in $C^*(E)\otimes\BK$.
\end{lemma}
\begin{proof}
Since $E$ is stably complete, there exists a regular vertex $w$ with $v\dm w$ and $w\dm v$.  Since $w$ is regular, it supports a loop $l$.
Write $T=\{e_1,\ldots, e_t\}$.
Since $v$ emits infinitely to $w$, there exists $f_1,\ldots, f_t$ with $f_i\neq f_j$ when $i\neq j$, $s_E(f_i)=v$, $r_E(f_i)=w$ and $f_i\notin T$ for all $i\in\{1,\ldots,t\}$.
Put $F=\{f_1,\ldots,f_t\}$.
Then 
\[ p_v-\sum_{e\in T}s_es_e^* = \left(p_v-\sum_{e\in T\cup F}s_es_e^*\right) + \sum_{e\in F}s_es_e^* \]
in $C^*(E)$.
Now, let $\{e_{i,j}\}_{i,j}$ denote a system of matrix units for $\BK$, and note since $s_{f_i}s_{f_i}^*\sim s_ls_l^*$ for all $i$ that
\begin{align*}
\sum_{e\in F}s_es_e^*\otimes e_{1,1} &\sim \sum_{i=1}^t p_w\otimes e_{i+1,i+1} \\
&= \sum_{i=1}^t\sum_{e\in s_E^{-1}(w)} s_es_e^*\otimes e_{i+1,i+1} \\
&\sim \sum_{i=1}^t \left( s_{f_i}s_{f_i}^* + \sum_{e\in s_E^{-1}(w)\setminus\{l\}} s_es_e^*\right)\otimes e_{i+1,i+1} \\
&\sim \sum_{i=1}^t \left( s_{f_i}s_{f_i}^*\otimes e_{1,1} + \sum_{e\in s_E^{-1}(w)\setminus\{l\}} s_es_e^*\otimes e_{i+1,i+1}\right) \\
&= \sum_{e\in F}s_es_e^*\otimes e_{1,1} + \sum_{i=1}^t\sum_{e\in s_E^{-1}(w)\setminus\{l\}} s_es_e^*\otimes e_{i+1,i+1} .
\end{align*}
Note for all $i\in\{1,\ldots, t\}$ that $w\dm r_E(e_i)$ and thereby that there exists an edge $g_i$ with $s_E(g_i)=w$ and $r_E(g_i)=r_E(e_i)$.
Then $s_{g_i}s_{g_i}^*\sim s_{e_i}s_{e_i}^*$ for all $i$, and thereby
\[
\sum_{e\in F}s_es_e^*\otimes e_{1,1} \sim \sum_{e\in T\cup F}s_es_e^*\otimes e_{1,1} + \sum_{i=1}^t\sum_{e\in s_E^{-1}(w)\setminus\{l,g_i\}} s_es_e^*\otimes e_{i+1,i+1} .
\]
Since both of the two equivalent projections above are orthogonal to \[\left(p_v-\sum_{e\in T\cup F}s_es_e^*\right)\otimes e_{1,1},\] we conclude that
\[ \left(p_v-\sum_{e\in T}s_es_e^*\right)\otimes e_{1,1} \sim p_v\otimes e_{1,1} + \sum_{i=1}^t\sum_{e\in s_E^{-1}(w)\setminus\{l,g_i\}} s_es_e^*\otimes e_{i+1,i+1} \]
 which is the desired as $s_es_e^*\sim p_{r_E(e)}$ for all edges $e$.
\end{proof}

\begin{lemma} \label{l:infemit_fin_dom}
Let $E$ be a stably complete graph and let $\{p_u, s_e \mid u\in E^0, e\in E^1\}$  denote a universal Cuntz--Krieger family generating $C^*(E)$.
Let $v$ be an infinite emitter in $E$ not supporting a loop, and assume that there exists a regular vertex $w$ with $w\dm v$.
Let $n$ be an integer, and let $T$ be a finite subset of $s_E^{-1}(v)$.
Then there exists a family $\{n_u\}_{w\dm u}$ of nonnegative integers for which
\[ p_w\oplus  n\left(p_v-\sum_{e\in T}s_es_e^*\right) \sim p_w\oplus np_v\oplus \bigoplus_{w\dm u} n_up_u \]
in $C^*(E)\otimes\BK$.
\end{lemma}
\begin{proof}
It suffices to prove the lemma for $n=1$.
Since $w$ is regular, it supports a loop $l$.
Hence $s_ls_l^*\sim p_w$.
Write $T=\{e_1,\ldots,e_t\}$, and let $\{e_{i,j}\}_{i,j}$ denote a system of matrix units for $\BK$.
Since $w\dm r_E(e_i)$, there exists for each $i$ an edge $g_i$ with $s_E(g_i)=w$ and $r_E(g_i)=r_E(e_i)$.
Note that $s_{g_i}s_{g_i}^*\sim s_{e_i}s_{e_i}^*$ for all $i$.
Hence,
\begin{align*}
p_w\otimes e_{1,1} &= \sum_{e\in s_E^{-1}(w)} s_es_e^*\otimes e_{1,1} \\
 &\sim p_w\otimes e_{1,1} + \sum_{e\in s_E^{-1}(w)\setminus\{l\}} s_es_e^*\otimes e_{t+1,t+1} \\
&\sim p_w\otimes e_{1,1} + \sum_{e\in s_E^{-1}(w)\setminus\{l\}} s_es_e^*\otimes e_{t,t} + \sum_{e\in s_E^{-1}(w)\setminus\{l\}} s_es_e^*\otimes e_{t+1,t+1} \\
&\sim p_w\otimes e_{1,1} +\sum_{i=1}^t \sum_{e\in s_E^{-1}(w)\setminus\{l\}} s_es_e^*\otimes e_{i+1,i+1} \\
&\sim p_w\otimes e_{1,1} +\sum_{i=1}^t\left( s_{e_i}s_{e_i}^*\otimes e_{i+1,i+1} + \sum_{e\in s_E^{-1}(w)\setminus\{l,g_i\}} s_es_e^*\otimes e_{i+1,i+1}\right) \\
&\sim p_w\otimes e_{1,1} + \sum_{e\in T}s_es_e^*\otimes e_{1,1} + \sum_{i=1}^t \sum_{e\in s_E^{-1}(w)\setminus\{l,g_i\}} s_es_e^*\otimes e_{i+1,i+1}.
\end{align*}
Since both the first and the last of the above equivalent projections are orthogonal to $\left(p_v-\sum_{e\in T}s_es_e^*\right)\otimes e_{1,1}$, we conclude that
\[ \left(p_w + \left(p_v-\sum_{e\in T}s_es_e^*\right)\right)\otimes e_{1,1} \sim \left(p_w + p_v\right)\otimes e_{1,1} + \sum_{i=1}^t\sum_{e\in s_E^{-1}(w)\setminus\{l,g_i\}} s_es_e^*\otimes e_{i+1,i+1}
 \]
 which is the desired as $s_es_e^*\sim p_{r_E(e)}$ for all edges $e$.
\end{proof}

\begin{lemma} \label{l:infemit_fin_nodom}
Let $E$ be a stably complete graph and let $\{p_u, s_e \mid u\in E^0, e\in E^1\}$  denote a universal Cuntz--Krieger family generating $C^*(E)$.
Let $v$ be an infinite emitter in $E$ not supporting a loop, and assume that there exists no regular vertex $w$ with $w\dm v$.
Let $T$ be a finite subset of $s_E^{-1}(v)$.
Then there exists a $\ast$-automorphism $\phi\colon C^*(E)\otimes\BK\to C^*(E)\otimes\BK$ with the following properties:
\begin{enumerate}
\item For all $u\in E^0\setminus\{v\}$ and all finite subsets $U\subset s_E^{-1}(u)$ with $v\notin r_E(U)$, \label{it:autone}
\[ \phi\left(p_u-\sum_{e\in U}s_es_e^*\right) \sim p_u-\sum_{e\in U}s_es_e^* ,\]

\item for all $u\in E^0\setminus\{v\}$ and all finite subsets $U\subset s_E^{-1}(u)$ with $v\in r_E(U)$, there exists a finite set $U'$ for which $U\subset U'\subset s_E^{-1}(u)$ and \label{it:auttwo}
\[ \phi\left(p_u-\sum_{e\in U}s_es_e^*\right) \sim p_u-\sum_{e\in U'}s_es_e^*, \]

\item and for all $T'\subset T$ there exists a family $\{n_u\}_{v\dm u}$ of nonnegative integers for which \label{it:autthree}
\[ \phi\left(p_v-\sum_{e\in T'}s_es_e^*\right) \sim p_v+ \bigoplus_{v\dm u} n_up_u .\]
\end{enumerate}
\end{lemma}
\begin{proof}
Define $\E_2=T$, and $\E_1=s_E^{-1}(v)\setminus\E_2$.
Let $E_\os$ denote the graph obtained from $E$ by out-splitting $v$ into $v^1$ and $v^2$ with respect to~$\E_1$ and~$\E_2$.
Since $v$ supports no loops and $\E_2$ is finite, the vertex $v^2$ of $E_\os$ is a regular vertex not supporting a loop.
Let $G$ denote the graph obtained from $E_\os$ by collapsing $v^2$.

Recall that $G$ is
\begin{align*}
G^0 &= E^0 \setminus \{v\} \cup \{v^1\} \\ 
G^1 &= \left( E^1\setminus \left(r_E^{-1}(v)\cup T\right)\right) \cup \{e^1 \mid e\in r_E^{-1}(v)\}\cup \{[e^2f]\mid e\in r_E^{-1}(v),f\in T\}
\end{align*}
with $r_G$ and $s_G$ defined as follows:
\begin{align*}
r_G(g) &= \begin{cases}
r_E(g) & \textnormal{when } g\in E^1\setminus \left(r_E^{-1}(v)\cup T\right) \\
v^1 & \textnormal{when } g=e^1, e\in r_E^{-1}(v) \\
r_E(f) & \textnormal{when } g=[e^2f_i], e\in r_E^{-1}(v), f\in T
\end{cases} \\
s_G(g)&= \begin{cases}
s_E(g) & \textnormal{when } g\in E^1\setminus \left(r_E^{-1}(v)\cup s_E^{-1}(v)\right) \\
v^1 & \textnormal{when } g\in s_E^{-1}(v)\setminus T \\
s_E(e) & \textnormal{when } g=e^1, e\in r_E^{-1}(v) \\
s_E(e) & \textnormal{when } g=[e^2f], e\in r_E^{-1}(v), f\in T
\end{cases}
\end{align*}

For all $f\in T$, there are infinitely many edges from $v$ to $r_E(f)$.  
Since no regular vertex emits to $v$, there are infinitely many edges from $s_E(e)$ to $r_E(f)$ for all $f\in T$ and all $e\in r_E^{-1}(v)$.  By renumbering these, one can therefore construct a graph isomorphism  $\Phi\colon G\to E$ with $\Phi^0(v^1)=v$ and $\Phi^0(u)=u$ for $u\in G^0\setminus\{v^1\}$.

Let $\{q_u, t_e \mid u\in G^0, e\in G^1\}$, $\{\bar p_v, \bar s_e \mid u\in E_\os^0, e\in E_\os^1\}$, and $\{p_u, s_e \mid u\in E^0, e\in E^1\}$  denote universal Cuntz--Krieger families generating $C^*(G)$, $C^*(E_\os)$, and $C^*(E)$ respectively.
Let $\Phi^*$ denote the \siso{} $C^*(G)\otimes\BK\to C^*(E)\otimes\BK$ induced by $t_e\mapsto s_{\Phi^1(e)}$ and $q_u\mapsto p_{\Phi^0(u)}$.
Let $\psi_\os\colon C^*(E)\to C^*(E_\os)$ denote the \siso{} given by Proposition~\ref{t:out-splitting-BP},
let $\psi_\col\colon C^*(E_\os)\otimes\BK\to C^*(G)\otimes\BK$ denote the \siso{} given by Proposition~\ref{t:collapsing}, and let $\psi$ denote the composite
\[ C^*(E)\otimes\BK \xrightarrow{\psi_\os\otimes id_{\BK}} C^*(E_\os)\otimes\BK \xrightarrow{\psi_\col} C^*(G)\otimes\BK . \]
Let $\phi$ denote the $\ast$-automorphism given as the composite $\Phi^* \circ \psi$.

Let $u\in E^0\setminus\{v\}$ and let $e\in s_E^{-1}(u)$.
Since $u\neq v$, $\psi_\col(\bar p_u)\sim q_u$ and thereby $\psi(p_u)\sim q_u$, hence $\phi(p_u)\sim p_u$.
Assume first that $r_E(e)\neq v$.
As $s_es_e^*\sim p_{r_E(e)}$ and $r_G(e)=r_E(e)$, $\psi(s_es_e^*)\sim t_et_e^*$ and thereby $\phi(s_es_e^*)\sim s_{\Phi^1(e)}s_{\Phi^1(e)}^*\sim s_es_e^*$.
This establishes~\eqref{it:autone}.
Assume now that $r_E(e)=v$.
Then 
\[ \psi_\os(s_es_e^*) = \bar s_{e^1}\bar s_{e^1}^* +\bar s_{e^2}\bar s_{e^2}^* = \bar s_{e^1}\bar s_{e^1}^* +\bar s_{e^2}\bar p_{v^2}\bar s_{e^2}^* = \bar s_{e^1}\bar s_{e^1}^* +\sum_{f\in T} \bar s_{e^2}\bar s_{f}\bar s_{f}^*\bar s_{e^2}^* \]
with $\psi_\col(\bar s_{e^1}s_{e^1}^*)\sim\psi_\col(\bar p_{v^1})\sim q_{v^1}\sim t_{e^1}t^*_{e^1}$ and
$\psi_\col(\bar s_{e^2}\bar s_{f}\bar s_{f}^*\bar s_{e^2}^*) \sim \psi_\col(\bar p_{r_E(f)})\sim q_{r_E(f)} \sim t_{[e^2f]}t_{[e^2f]}^*$ for all $f\in T$.
Since $u=s_E(e)\dm v$, $u$ is an infinite emitter.  So there are distinct edges $\{g_f\}_{f\in T}$ with $s_E(g_f)=u$ and $r_E(g_f)=r_E(f)$, and thereby $s_{g_f}s_{g_f}^*\sim s_{\Phi^1([e^2f])}s_{\Phi^1([e^2f])}^*$ for all $f\in T$.
Hence
\[ \phi(s_es_e^*)\sim s_{\Phi^1(e^1)}s^*_{\Phi^1(e^1)} + \sum_{f\in T} s_{\Phi^1([e^2f])}s_{\Phi^1([e^2f])}^* 
\sim s_es_e^* + \sum_{f\in T} s_{g_f}s_{g_f}^* \]
which together with~\eqref{it:autone} establishes~\eqref{it:auttwo}.

To prove~\eqref{it:autthree}, let $T'\subset T$, and fix an edge $g$ with $r_E(g)=v$.
Then
\begin{align*} 
\psi_\os\left(p_v-\sum_{e\in T'}s_es_e^*\right) &= \bar p_{v^1} + \bar p_{v^2} - \sum_{e\in T'} \bar s_e\bar s_e^* \\
&= \bar p_{v^1} + \sum_{e\in T\setminus T'}\bar s_e\bar s_e^* \\
&\sim \bar p_{v^1} + \sum_{e\in T\setminus T'}\bar s_{g^2}\bar s_e\bar s_e^*s_{g^2}^*,
\end{align*}
with $\psi_\col(\bar s_{g^2}\bar s_e\bar s_e^*s_{g^2}^*)\sim \psi_\col(\bar p_{r_E(e)})\sim q_{r_E(e)} \sim t_{[g^2e]}t_{[g^2e]}^*$ for all $e\in T\setminus T'$.
Hence
\[ \psi\left(p_v-\sum_{e\in T'}s_es_e^*\right)\sim q_{v^1} + \sum_{e\in T\setminus T'}t_{[g^2e]}t_{[g^2e]}^*  \]
with $s_{\Phi^1([g^2e])}s_{\Phi^1([g^2e])}^*\sim s_gs_es_e^*s_g^*$ for all $e\in T\setminus T$, implying that
\[ \phi\left(p_v-\sum_{e\in T'}s_es_e^*\right)\sim p_v + \sum_{e\in T\setminus T'}s_gs_es_e^*s_g^*  \]
which gives the desired as $s_gs_es_e^*s_g^*\sim p_{r_E(e)}$ with $v\dm r_E(e)$ for all $e\in T\setminus T'$.
\end{proof}

We write $p\gtrsim q$ 
for projections $p$ and $q$ in the \csa{} $\A$
if there exists a projection $q'\in\A\otimes\BK$ with $p\geq q'$ and $q'\sim q$.

\begin{lemma} \label{l:dominate}
Let $\A$ be a \csa{}, and
let $p$ and $q$ be projections in $\A\otimes\BK$. Assume that $p\gtrsim q$. Then $\infty p\sim q\oplus\infty p\sim\infty (q\oplus p)$ in $\Mul(\A\otimes\BK)$.
\end{lemma}

\begin{proof}
Set $P = \infty p$.  By Theorem~2.1 of \cite{hr:stable}, $\B = P ( \A \otimes \BK ) P$ is stable since $\{ s_k p s_k^* \}_{ k \in \N }$ is a collection of mutually orthogonal, mutually equivalent projections whose sum $\sum_ {k = 1}^\infty s_k p s_k^*$ converges strictly to $P$ in $\Mul ( \B )$.  So, for any projection $e$ in $\Mul ( \B )$ if $e \sim e'$ and $1_{\Mul(\B)} \sim P'$ in $\Mul(\B)$ such that $e'$ and $P'$ are orthogonal, then $e' + Q' \sim 1_{\Mul(\B)}$.  Since $\Mul ( \B) \cong P \Mul (\A \otimes \BK ) P$, we have that $e \oplus P \sim e' + P' \sim P$ for all projections $e$ in $P \Mul ( \A \otimes \BK ) P$.  The lemma now follows since $q\oplus \infty p \sim e \oplus P$ and $\infty ( q \oplus p) \sim f \oplus \infty p$ for some projections $e, f$ in $P \Mul ( \A \otimes \BK ) P$.  
\end{proof}

\begin{lemma} \label{l:infemit_inf}
Let $E$ be a stably complete graph and let $\{p_u, s_e \mid u\in E^0, e\in E^1\}$  denote a universal Cuntz--Krieger family generating $C^*(E)$.  Let $\{q_k\}_{k=1}^\infty$ be a sequence of projections in $C^*(E)\otimes\BK$ and fix a \full{} \partitioned{} sequence $\{\{n^{(k)}_{(v,T)}\}_{(v,T)\in S_k}\}_{k=1}^\infty$ of \gc{E} for $\{q_k\}_{k=1}^\infty$.

Assume that the sum $\sum_{k=1}^\infty q_k$ converges strictly to a projection $p$ in $\Mul(C^*(E)\otimes\BK)$.
Define for all $v\in E^0$
\[ T_v = \bigcup_{k=1}^\infty\bigcup_{(v,T)\in S_k} T, \]
and assume for all $v\in E^0$ that either $T_v=\emptyset$ or there exists $w\in E^0$ with $w\geq v$ for which $T_w$ is infinite.
For all $v\in E^0$ and $k\in\N$, we define $n_{(v,\emptyset)}^{(k)}=0$ if $(v,\emptyset)\notin S_k$.
Define for all $v\in E^0$,
\[
n_v = \begin{cases}
 \sum_{k=1}^\infty n_{(v,\emptyset)}^{(k)} & \textnormal{if } T_w=\emptyset\textnormal{ for all } w\geq v 
 \\
 \infty & \textnormal{if } T_v\textnormal{ is infinite and } T_w=\emptyset\textnormal{ for all } w\in E^0\setminus\{v\} \textnormal{ with } w\geq v  \\
 1 & \textnormal{otherwise}. \end{cases}
\]

Then the sum $\sum_{v\in E^0} n_vp_v$ is a projection in $\Mul(C^*(E)\otimes\BK)$ Murray--von Neumann equivalent to $p$.
\end{lemma}

\begin{proof}
Define a partition $\{A,B,C\}$ of $E^0$ as follows.  Let $A$ denote the set of $v\in E^0$ for which $T_w=\emptyset$ for all $w\geq v$, let $B$ denote the set of vertices $v$ for which $T_v$ is infinite and $T_w=\emptyset$ for all $w>v$, and let $C$ denote $E^0\setminus(A\cup B)$.
For $v\in C$, we let $\kappa_v\in\N\cup\{\infty\}$ denote the number of elements in $\{ T\subset s_E^{-1}(v) \mid (v,T)\in\bigcup_{k=1}^\infty S_k\}$, and for $v\in B\cup C$ we renumber $\{ (T,n^{(k)}_{v,T}) \mid k\in\N, T\subset s_E^{-1}(v), (v,T)\in S_k\}$ as $\{(T_k^v,n_k^v)\}$, so that we may write
\begin{align*}
p &= \sum_{k=1}^\infty q_k \sim \bigoplus_{k=1}^\infty q_k \sim
\bigoplus_{k=1}^\infty\bigoplus_{(v,T)\in S_k} n_{(v,T)}^{(k)} \left(p_v-\sum_{e\in T}s_es_e^*\right) \\
&\sim \bigoplus_{v\in A} n_vp_v
\oplus \bigoplus_{v\in B} \bigoplus_{k=1}^\infty n^v_k\left(p_v-\sum_{e\in T_k^v}s_es_e^*\right)
\oplus \bigoplus_{v\in C}  \bigoplus_{k=1}^{\kappa_v} n^v_k\left(p_v-\sum_{e\in T_k^v}s_es_e^*\right).
\end{align*}
Let $v\in B$.  Choose $F=\{f_1,\ldots,f_{|r_E(T_v)|}\}\subset s_E^{-1}(v)\setminus T_v$ such that $r_E(T_v)=r_E(F)$.
Consider $\bigoplus_{k=1}^\infty n^v_k\left(p_v-\sum_{e\in T_k^v}s_es_e^*\right)$ and let us show that it is Murray--von Neumann equivalent to $\infty p_v$.
Note for fixed $k\in\N$ that
\[
p_v-\sum_{e\in T_k^v}s_es_e^* 
\sim \left(p_v-\sum_{e\in T_k^v\cup F}s_es_e^*\right) \oplus s_{f_1}s_{f_1}^* \oplus\cdots\oplus s_{f_{|r_E(T_v)|}}s_{f_{|r_E(T_v)|}}^* 
\]
where we let the latter projection in $C^*(E)\otimes\BK$ be denoted by $r_k$.
Note that each $s_{f_i}s_{f_i}^*$ appears infinitely often in the direct sum $\bigoplus_{k=1}^\infty n_k^vr_k$.
By rearranging the blocks we can therefore get $\bigoplus_{k=1}^\infty n_k^vr_k\sim \bigoplus_{k=1}^\infty n_k^v\bar r_k$ with each projection $\bar r_k$ defined as
\[
\left(p_v-\sum_{e\in T_k^v\cup F}s_es_e^*\right) \oplus s_{f_1}s_{f_1}^* \oplus\cdots\oplus s_{f_{|r_E(T_v)|}}s_{f_{E(T_v)|}}^*\oplus s_{f_{i_1}}s_{f_{i_1}}^*\oplus\cdots\oplus s_{f_{i_{n(k)}}}s_{f_{i_{n(k)}}}^*
\]
where $i_1,\ldots,i_{n(k)}$ are chosen such that $n(k)=|T_k^v|$ and $r_E(f_{i_j})=r_E(e_j)$ when $T_k^v=\{e_1,\ldots,e_{n(k)}\}$.  Since $s_{f_{i_j}}s_{f_{i_j}}^*\sim s_{e_j}s_{e_j}^*$, we note that $\bar r_k \sim p_v$.
Hence $\bigoplus_{k=1}^\infty n^v_k\left(p_v-\sum_{e\in T_k^v}s_es_e^*\right)\sim \infty p_v$.

Let $v\in C$.  Choose a $w\in B$ for which $w>v$, and consider the projection $r_v$ defined as
\[ r_v = \infty p_w \oplus  \bigoplus_{k=1}^{\kappa_v} n^v_k\left(p_v-\sum_{e\in T_k^v}s_es_e*\right). \]
Since $p_w\gtrsim s_es_e^*$ for all $e\in T_k^v$, we note that
\[ \infty p_w \sim \infty p_w \oplus\bigoplus_{k=1}^{\kappa_v}\bigoplus_{e\in T_k^v}n_k^vs_es_e^* \]
so $r_v\sim \infty p_w\oplus\bigoplus_{e\in T_k^v}n_k^vp_v$, and we
 conclude that $r_v\sim \infty p_w\oplus p_v$ since $p_w\gtrsim p_v$, cf.~Lemma~\ref{l:dominate}.

We may therefore conclude that
\begin{align*}
 p&\sim \bigoplus_{v\in A} n_vp_v
\oplus \bigoplus_{v\in B} \infty p_v
\oplus \bigoplus_{v\in C}  \bigoplus_{k=1}^{\kappa_v} n^v_k\left(p_v-\sum_{e\in T_k^v}s_es_e*\right) \\
&\sim \bigoplus_{v\in A} n_vp_v
\oplus \bigoplus_{v\in B} \infty p_v
\oplus \bigoplus_{v\in C}  r_v \\
&\sim \bigoplus_{v\in A} n_vp_v
\oplus \bigoplus_{v\in B} \infty p_v
\oplus \bigoplus_{v\in C}  p_v
\end{align*}
as desired.
\end{proof}

\begin{remark}
Let $E$ be a graph with finitely many vertices.
Let $p$ be a projection in $\Mul(C^*(E)\otimes\BK)$ and assume that $p$ is Murray--von Neumann equivalent to a direct sum $\bigoplus_{v\in E^0}n_vp_v$ with $n_v\in\N\cup\{\infty\}$ for all $v\in E^0$. Let $\{m_v\}_{v\in E^0}$ with $m_v\in\N\cup\{\infty\}$ for all $v\in E^0$, and assume for all $v\in E^0$ that $n_v\neq m_v$ only occurs if $n_w=\infty$ for some $w\in E^0$ with $w>v$.  Then $\bigoplus_{v\in E^0}m_vp_v$ is Murray--von Neumann equivalent to $p$.
\end{remark}

\section{Corners of unital graph $C^*$-algebras}

For a graph $E$, the \emph{stabilization} $SE$ of $E$ is the graph obtained by adding an infinite head to every vertex of $E$:
\begin{align*}
(SE)^0 &= \{ v^i \mid v\in E^0, i\in\N_0\}, \\
(SE)^1 &= E^1 \cup \{ f(v)^i \mid v\in E^0, i\in\N \},
\end{align*}
with range and source defined by $r_{SE}(e)=r_E(e)^0$ and $s_{SE}(e)=s_E(e)^0$ for $e\in E^1$, and $r_{SE}(f(v)^i)=v^{i-1}$ and $s_{SE}(f(v)^i)=v^i$ for $v\in E^0$ and $i\in\N$.
We have the following proposition.

\begin{proposition}\label{p:stablegraph}
 Let $E$ be a graph, $SE$ be the stabilization of $E$, and let $\mathbb K$ be the compact operators with standard matrix units $\{e_{i,j}\}_{i,j=0}^\infty$.
 There is a $\ast$-isomorphism $\phi \colon C^\ast(SE) \to C^\ast(E) \otimes \mathbb K$ given by
 \begin{eqnarray*}
  p_{v^i} &\mapsto& p_v \otimes e_{i,i}, \quad \text{for }v\in E^0, i \in \N_0, \\
  s_e & \mapsto & s_e \otimes e_{0,0}, \quad \text{for } e \in E^1, \\
  s_{f(v)^i} & \mapsto & p_v \otimes e_{i,i-1}, \quad \text{for } v\in E^0, i\geq 1.
 \end{eqnarray*}
\end{proposition}
\begin{proof}
 It is straight forward to verify that the above assignment satisfies the Cuntz--Krieger relations and thus induces a \siso{}. We will construct the inverse.
 
 Let $\psi_1 \colon C^\ast(E) \to \Mul(C^\ast(SE))$ be given by
 \begin{eqnarray*}
  p_v & \mapsto & \sum_{i=0}^\infty p_{v^i} , \quad \text{for } v\in E^0 \\
  s_e &\mapsto & s_e + \sum_{i=1}^\infty s_{f(s(e))^i f(s(e))^{i-1} \dots f(s(e))^2 e} s^\ast_{f(r(e))^i \dots f(r(e))^2}, \quad \text{for } e\in E^1
 \end{eqnarray*}
 where the sums converge in the strict topology. It is straight forward to verify that this is well-defined, i.e.~that the above projections and partial isometries satisfy the Cuntz--Krieger relations.
 
 Let $\psi_2 \colon \mathbb K \to \Mul(C^\ast(SE))$ be given by
 \[
  e_{i,i-1} \mapsto \sum_{v\in E^0} s_{f(v)^i}, \quad \text{for } i \geq 1,
 \]
 with the sum converging strictly. It is straight forward to check that $\psi_2$ is well-defined, that the image of $\psi_2$ commutes with the image of $\psi_1$, and that
 \[ \psi_1(C^\ast(E)) \psi_2(\mathbb K) \subset C^\ast(SE).\]
 Thus there is an induced $\ast$-homomorphism $\psi \colon C^\ast(E) \otimes \mathbb K \to C^\ast(SE)$ given by corestricting $\psi_1 \times \psi_2 \colon C^\ast(E) \otimes \mathbb K \to \Mul(C^\ast(SE))$ to $C^\ast(SE)$.
 Straight forward computations show that $\phi$ and $\psi$ are each others inverses.
\end{proof}

The following theorem is proved in the same way as Theorem~3.15 of \cite{ar-corners-ck-algs}.

\begin{theorem}\label{t:fullcorners-stabilized}
Let $E$ be a graph and let $T$ be a hereditary subset of $(SE)^{0}$ such that $E^{0} \subset T$.  Define $p_T$ as the strict limit in the multiplier algebra $\Mul(C^*(SE))$
\begin{align*}
p_{T} = \sum_{ v \in T } p_{v}
\end{align*}
where
$\{ s_{e} , p_{v} \ : \ e \in (SE)^{1} , v\in ( SE )^{0} \}$ is a universal Cuntz--Krieger $SE$-family generating $C^{*} (SE)$.
Then $p_{T} C^{*} (SE) p_{T}$ is a full hereditary \cssa{} of $C^{*} (SE)$, and $C^{*} (G) \cong p_{T}C^{*} ( SE )p_{T}$
holds for the graph $G=(T,s_{SE}^{-1}(T), r_{SE}, s_{SE})$.
Consequently, if $T$ is finite, then $p_{T}$ is a full projection in $C^{*} (SE)$.
\end{theorem}

\begin{proposition} \label{p:fullher}
Let $E$ be a stably complete graph. Let $\{q_k\}_{k=1}^\infty$ be a sequence of orthogonal projections in $C^*(E)\otimes\BK$ with $\sum_{k=1}^\infty q_k$ converging strictly to a projection $p$ in $\Mul(C^*(E)\otimes\BK)$.
If $\A=p(C^*(E)\otimes\BK)p$ is a full \cssa{} of $C^*(E)\otimes\BK$, then $\A$ is isomorphic to a graph \csa{}.  Furthermore, there exists a graph $G$ with $E\CKsub G\CKsub SE$ and $\A\cong C^*(G)$.
\end{proposition}
\begin{proof}
Let $\{p_v,s_e\mid v\in E^0,s\in E^1\}$ denote a universal Cuntz--Krieger $E$-family generating $C^*(E)$.
We first note that since $\A$ is a full \cssa{} of $C^*(E)\otimes\BK$, we may assume that $q_1$ is a full projection of $C^*(E)\otimes\BK$.
In fact, since $\Prim (\A) \cong \Prim \left( C^\ast(E) \right)$ is compact, and $\{ \sum_{k=1}^n q_k\}_{k=1}^\infty$ is an (increasing) approximate identity for $\A$, $\sum_{k=1}^n q_k$ will be full in $\A$ for large $n$, and thus also full in $C^\ast(E)\otimes \mathbb K$.
By replacing $\{q_k\}_{k=1}^\infty$ with $\sum_{k=1}^nq_i,q_{n+1},q_{n+2},\ldots$ the desired is achieved.

Let $\{\{n^{(k)}_{(v,T)}\}_{(v,T)\in S_k}\}_{k=1}^\infty$ be a sequence of \gc E  for $\{q_k\}_{k=1}^\infty$.
By Lemma~\ref{l:fullproj}, we may assume since $q_1$ is full that for all $v\in E^0$ there exists $T\subset s_E^{-1}(v)$ with $(v,T)\in S_1$. Hence $\{n^{(k)}_{(v,T)}\}_{(v,T)\in S_k}$ is \full{}.
By Lemma~\ref{l:partitioned}, we may assume that $\{n^{(k)}_{(v,T)}\}_{(v,T)\in S_k}$ is \partitioned{} and \full{}.

Define for all $v\in E^0$,
\[ T_v = \bigcup_{k=1}^\infty\bigcup_{(v,T)\in S_k}T .\]
We first consider infinite emitters $v$ with the following properties: $v$ does not support a loop, there exists a regular vertex $w$ with $w\dm v$, and $T_v$ is finite and nonempty.
For such $v$, we need the existence of a regular vertex $w\dm v$ for which $(w,\emptyset)\in S_k$ whenever there is a $T\neq\emptyset$ with $(v,T)\in S_k$.
To see that we may assume this, let $v$ be such an infinite emitter.
Define $N$ as the largest integer $k$ for which there exists $T\neq\emptyset$ with $(v,T)\in S_k$.
Replace $\{q_k\}_{k=1}^\infty$ with $q_1+\cdots +q_N,q_{N+1}, q_{N+2},\ldots$ and their \full{} \partitioned{} \gc E with
$\{\{\overline n^{(k)}_{(u,T)}\}_{(u,T)\in \overline S_k}\}_{k=N}^\infty$ defined by
\[ \bigoplus_{(u,T)\in \overline S_N}\overline n^{(N)}_{(u,T)}\left(p_u-\sum_{e\in T}s_es_e^*\right) = \bigoplus_{k=1}^N\bigoplus_{(u,T)\in S_k}n^{(k)}_{(u,T)}\left(p_u-\sum_{e\in T}s_es_e^*\right) \]
and $\overline S_k=S_k$ and $\overline n^{(k)}_{(u,T)} = n^{(k)}_{(u,T)}$  for $k>N$  and $(u,T)\in\overline S_k$.
Then $(v,T)\in \overline S_k$ only occurs when $T=\emptyset$ or $k=N$.
By assumption, there exists a regular vertex $w$ with $w\dm v$.  Since $(w,\emptyset)\in S_1$, we see that $(w,\emptyset)\in \overline S_N$, as desired.
By doing this recursively for the finitely many infinite emitters $v$ of this type, we achieve the desired.

We are now in the situation where all infinite emitters $v$ fall into at least one of the following categories:
\begin{enumerate}
\item $v$ supports a loop; \label{it:infemit_loop}
\item $v$ does not support a loop, $T_v$ is finite, and there exists a regular vertex $w\dm v$ for which $(w,\emptyset)\in S_k$ whenever there is a $T\neq\emptyset$ with $(v,T)\in S_k$; \label{it:infemit_fin_dom}
\item $v$ does not support at loop,  there is no regular vertex $w\dm v$, and $T_v$ is finite; \label{it:infemit_fin_nodom}
\item $T_v$ is infinite. \label{it:infemit_inf}
\end{enumerate}

For each infinite emitter $v$ in category~\eqref{it:infemit_loop}, we can apply Lemma~\ref{l:infemit_loop} to all $T\neq\emptyset$ with $(v,T)\in S_k$ for some $k$.  By replacing the \full{} \partitioned{} \gc E for $\{q_k\}_{k=1}^\infty$ accordingly, we see that for such $v$ we may assume that $T_v$ is empty.

Similarly, we can apply Lemma~\ref{l:infemit_fin_dom} to all infinite emitters $v$ that fall in category~\eqref{it:infemit_fin_dom}.  So for such $v$ we may also assume that $T_v$ is empty.

Let $v$ be an infinite emitter in category~\eqref{it:infemit_fin_nodom}.
Let $\phi\colon C^*(E)\otimes\BK\to C^*(E)\otimes\BK$ denote the \siso{} given by Lemma~\ref{l:infemit_fin_nodom}.  By replacing $\{q_k\}_{k=1}^\infty$ with $\{\phi(q_k)\}_{k=1}^\infty$
we may assume that $T_v$ is empty.
More concretely, let for each $k\in\N$, $\{\overline n^{(k)}_{(u,T)}\}_{(u,T)\in \overline S_k}$ be \gc E for $\phi(q_k)$ achieved by applying \eqref{it:autone}--\eqref{it:autthree} of Lemma~\ref{l:infemit_fin_nodom} to
\[ \phi(q_k) \sim \bigoplus_{(u,T)\in S_k} \phi\left(p_u-\sum_{e\in T}s_es_e^*\right) . \]
Clearly, $\{\{\overline n^{(k)}_{(u,T)}\}_{(u,T)\in \overline S_k}\}_{k=1}^\infty$ is \full{} and \partitioned{} since $\{\{ n^{(k)}_{(u,T)}\}_{(u,T)\in S_k}\}_{k=1}^\infty$ is.
Let $ \overline T_u = \bigcup_{k=1}^\infty\bigcup_{(u,T)\in \overline S_k}T$.
Now, by \eqref{it:autone}--\eqref{it:auttwo} of Lemma~\ref{l:infemit_fin_nodom}, all infinite emitters $w\in E^0\setminus\{v\}$ fall into category \eqref{it:infemit_fin_nodom} if $\overline T_w$ is finite.
And by \eqref{it:autthree} of Lemma~\ref{l:infemit_fin_nodom}, $\overline T_v$ is empty.

We have now achieved
that $T_v\neq\emptyset$ only occurs in category~\eqref{it:infemit_inf}, i.e., when $T_v$ is infinite.
Define for all $v\in E^0$
\[
n_v = \begin{cases}
 \sum_{k=1}^\infty n_{(v,\emptyset)}^{(k)} & \textnormal{if } T_w=\emptyset\textnormal{ for all } w\geq v \\
 \infty & \textnormal{if } T_v\textnormal{ is infinite and } T_w=\emptyset\textnormal{ for all } w>v \\
 1 & \textnormal{otherwise}. \end{cases}
\]
By Lemma~\ref{l:infemit_inf}, $p\sim\bigoplus_{v\in E^0}n_vp_v$ in $\Mul(C^*(E)\otimes\BK)$.

Define $T\subset (SE)^0$ by
\[ T=\{v^i \mid v\in E^0, i\in\{0,\ldots,n_v-1\}\}, \]
and $p_T=\sum_{v\in T}p_v$.
Let $\overline{\psi} \colon \Mul(C^\ast(SE)) \to \Mul(C^\ast(E)\otimes \mathbb K)$ be the \siso{} induced by Proposition \ref{p:stablegraph}.
It is easily seen (by the definition of $\psi$) that $\overline{\psi}^{-1}(\bigoplus_{v\in E^0}n_vp_v)$ is Murray--von Neumann equivalent to $p_T$.
Hence $p_TC^*(SE)p_T\cong p(C^*(E)\otimes\BK)p$.
So by Theorem~\ref{t:fullcorners-stabilized}, $\A\cong C^*(G)$ for the CK-subgraph~$G$ of~$SE$ defined by
\[ G= (T,s_{SE}^{-1}(T),r_{SE},s_{SE}) .\]
\end{proof}

\section{Hereditary $C^*$-subalgebras of unital graph $C^*$-algebras}
A \csa{} $\A$ is called \spu{} if it contains a countable approximate unit of projections.
All graph \csa{}s $C^*(E)$ are \spu{} as the finite sums of vertex projections $p_v$ constitutes an approximate unit of projections.

Recall that if $E$ is a graph with finitely many vertices, then by Proposition~\ref{p:canonical} there exists a stably complete graph $G$ such that $E\sim_MG$ and thereby $C^*(E)\otimes\BK\cong C^*(G)\otimes\BK$.

\begin{theorem} \label{t:main}
Let $E$ be a graph with finitely many vertices, and let $\A$ be a \csa{} strongly Morita equivalent to $C^*(E)$.  Then the following are equivalent:
\begin{enumerate}
\item \label{main:1} For any stably complete graph $G$ with $C^*(E) \otimes \BK \cong C^*(G) \otimes\BK$, there exists a graph $F$ with $G\CKsub F\CKsub SG$ and $\A\cong C^*(F)$.

\item \label{main:2}  $\A$ is isomorphic to a graph $C^*$-algebra.

\item \label{main:3} $\A$ is \spu.
\end{enumerate}
\end{theorem}

\begin{proof}
\eqref{main:1} $\implies$ \eqref{main:2} and \eqref{main:2} $\implies$ \eqref{main:3} are clear.  Assume that $\A$ is \spu{} and let $\{p_n\}_{n=1}^\infty$ denote an approximate unit of projections for $\A$.
Since $\A$ is \su{} and strongly Morita equivalent to $C^*(E)$ which is separable, $\A$ is stably isomorphic to $C^*(E)$ and therefore separable.
Since $\A$ is separable, we may assume that $\{p_n\}_{n=1}^\infty$ is increasing.

Let $G$ be a stably complete graph with $C^*(E) \otimes \BK \cong C^*(G) \otimes\BK$.  Let a \siso{} $\phi\colon\A\otimes\BK\to C^*(G)\otimes\BK$ be given and let $\{e_{i,j}\}_{i,j}$ denote a system of matrix units for $\BK$.
Define mutually orthogonal projections $q_n$ in $C^*(G)\otimes\BK$ by $q_n=\phi((p_n-p_{n-1})\otimes e_{1,1})$.
Note that $\{p_n\otimes e_{1,1}\}$ converges in the strict topology on $\Mul(\A\otimes\BK)$ to a projection $p$ with $p(\A\otimes\BK)p\cong\A$ being full in $\A\otimes\BK$.
Hence $\sum_{n=1}^\infty q_n$ converges in the strict topology on $\Mul(C^*(G)\otimes\BK)$ to a projection $q$ with $q(C^*(E)\otimes\BK)q\cong p(\A\otimes\BK)p$ and $q(C^*(G)\otimes\BK)q$ full in $C^*(E)\otimes\BK$.
By Proposition~\ref{p:fullher}, there exists a graph $F$ with $G\CKsub F\CKsub SG$ and $C^*(F)\cong q(C^*(G)\otimes\BK)q$.
Hence $\A\cong C^*(F)$.
\end{proof}

\begin{remark}
Note that if $E$ is not stably complete, there may not exist a graph $F$ with $E\CKsub F\CKsub SE$ and $\A\cong C^*(F)$: If $E^0=\{v,w\}$ and $E^1=\{f,g\}$ with $s_E(f)=r_E(g)=v$ and $r_E(f)=s_E(g)=w$ then $\A=C(S^1)$ is stably isomorphic to $C^*(E) \cong \mathsf{M}_2 ( C(S^1) )$ but $C^*(F)$ is nonabelian for all $F$ with $E\CKsub F\CKsub SE$.
\end{remark}

\begin{example}\label{e:spu}
Set $\A{} = \left\{ f \in C( S^{1} , \mathsf{M}_{2} ) :  f(1) \in \begin{bmatrix} \C & 0 \\ 0 & 0 \end{bmatrix} \right\}$.  Then $\A{}$ is a hereditary full \cssa{} of $C( S^{1} , \mathsf{M}_{2} )$ which
is not \spu{}.  Consequently, $\A$ is stably isomorphic to a unital graph \csa{} but is not isomorphic to a graph \csa{}.
\end{example}

\begin{example}\label{e:unital}
The UHF algebra $\Mat_{2^\infty}$ is not isomorphic to a graph \csa{} as the $K$-theory of a unital graph \csa{} is always finitely generated.
Since $\Mat_{2^\infty}\otimes\BK$ is isomorphic to a graph \csa{}, we conclude that Theorem~\ref{t:main} fails if one omits the assumption that the graph $E$ has finitely many vertices.
\end{example}

\begin{corollary}\label{c:hereditary}
Let $E$ be a graph with finitely many vertices, and let $\A$ be a hereditary \cssa{} of $C^*(E)\otimes\BK$.  Then the following are equivalent:
\begin{enumerate}
\item \label{her:1}  For any stably complete graph $G$ such that $C^*(E) \otimes \BK \cong C^*( G) \otimes \BK$,  there exists a hereditary subset $H$ of $G^0$ and a graph $F$ with $G_H\CKsub F\CKsub S(G_H)$ and $\A\cong C^*(F)$.

\item \label{her:2} $\A$ is isomorphic to a graph \csa{}.

\item \label{her:3} $\A$ is \spu{}.
\end{enumerate}
\end{corollary}

\begin{proof}
\eqref{her:1} $\implies$ \eqref{her:2} and \eqref{her:2} $\implies$ \eqref{her:3} are clear.  Assume that $\A$ is \spu{}.  Let $G$ be a stably complete graph with $C^*(E) \otimes \BK \cong C^*( G) \otimes \BK$.  Suppose $\A$ is a full hereditary \cssa{} of $C^*(E)\otimes\BK$.  Then by Corollary~2.9 of~\cite{b:hereditary}, $\A$ and $C^*(E)$ are stably isomorphic.  Then by Theorem~\ref{t:main}, there exists a graph $F$ with $G \CKsub F \CKsub SG$ with $\A \cong C^*(F)$.  Set $H = G^0$.

Assume that $\A$ is not full in $C^*(E)\otimes\BK$.  We may view $\A$ as a hereditary \cssa{} of $C^*(G)\otimes\BK$ since $C^*(E) \otimes \BK \cong C^*( G) \otimes \BK$.
Let $\I$ denote the ideal in $C^*(G)\otimes\BK$ generated by $\A$.
Let $\{p_n\}_{n=1}^\infty$ denote an increasing approximate unit of projections for $\A$, and let $\I_n$ denote the ideal in $C^*(G)\otimes\BK$ generated by~$p_n$.
By Corollary~\ref{c:non-stable-kthy}, $\I_n$ is gauge-invariant.
Since $G$ has finitely many vertices, $C^*(G)\otimes\BK$ has finitely many gauge-invariant ideals.
So we conclude that $\I=\I_n$ for some $n$, hence $\I$ is gauge-invariant and there exists an admissible pair $(H,S)$ in $G^0$ with $\I=I_{(H,S)}\otimes\BK$. 
Since $G$ is a stably complete graph, $G$ has no breaking vertices, hence $S=\emptyset$.
Furthermore, one easily checks that $G_H$ is a stably complete graph.
By Proposition~\ref{prop: full corner gauge invariant}, $\I=I_{H} \otimes \BK$ is isomorphic to a full hereditary \cssa{} of $C^* ( G_H) \otimes \BK$.  
We conclude that $\A$ is isomorphic to a full hereditary \cssa{} of $C^*(G_H)\otimes\BK$.  By Theorem~\ref{t:main}, there exists a graph $F$ with $G_H \CKsub F \CKsub S(G_H)$ such that $\A \cong C^*( F)$.
\end{proof}

\begin{corollary}
Let $E$ be a graph with finitely many vertices and assume that $E$ satisfies Condition~(K).
If $\A$ is a hereditary \cssa{} of $C^*(E)\otimes\BK$, then $\A$ is isomorphic to a graph \csa{}.
\end{corollary}
\begin{proof}
Since $E$ satisfies Condition~(K), $C^*(E)$ and thereby also $C^*(E)\otimes\BK$ have real rank zero.  Consequently, all hereditary \cssa{}s of $C^*(E)\otimes\BK$ are \spu{}.  In particular, $\A$ is \spu{}, and the desired follows from Corollary~\ref{c:hereditary}.
\end{proof}

\begin{corollary}
Let $E$ be a graph with finitely many vertices and assume that $E$ satisfies Condition~(K).
If $\A$ is a \csa{} stably isomorphic to $C^*(E)$, then $\A$ is isomorphic to a graph \csa{}.
\end{corollary}

\section{Unitizations and graph \csa{}s}
For a non-unital \csa{} $\A$, we let $\A^\dagger$ denote the minimal unitization of $\A$.  For $\A$ unital, we let $\A^\dagger=\A$.

\begin{definition}
Let $E$ be a graph, and let $H$ be a hereditary subset of $E^0$.  Consider the set
\[ F(H) = \{ \alpha\in E^* \mid \alpha = e_1\cdots e_n, s_E(e_n)\notin H, r_E(e_n)\in H \}, \]
 let $\overline F(H)$ denote a copy of $F(H)$, and write  $\overline\alpha$ for the copy of $\alpha$ in $\overline F(H)$. Define the graph $E(H)$ by
 \begin{align*}
 E(H)^0 &= H\cup F(H) \\
  E(H)^1 &= s_E^{-1}(H)\cup\overline F(H)
 \end{align*}
and by extending $s_E$ and $r_E$ to $E(H)$ by defining $s_{E(H)}(\overline\alpha)=\alpha$ and $r_{E(H)}(\overline\alpha)=r_E(\alpha)$.
\end{definition}

\begin{definition}
Let $E$ be a graph, and let $H$ be a hereditary subset of $E^0$.  Consider again the set
\[ F(H) = \{ \alpha\in E^* \mid \alpha = e_1\cdots e_n, s_E(e_n)\notin H, r_E(e_n)\in H \}. \]
Define the graph $\widetilde E(H)$ by
 \begin{align*}
 \widetilde E(H)^0 &= H\cup \{\star\} \\
 \widetilde E(H)^1 &= s_E^{-1}(H)\cup  F(H)
 \end{align*}
and by extending $s_E$ and $r_E$ to $E(H)$ by defining $s_{\widetilde E(H)}(\alpha)=\star$ and $r_{\widetilde E(H)}(\alpha)=r_E(\alpha)$.
\end{definition}
Theorem~3.9 of~\cite{ar-corners-ck-algs} and its proof hold in the following more general setting.

\begin{theorem} \label{t:EH}
Let $E$ be a graph and let $H$ be a hereditary subset of $E^0$.  Assume that the graph
\[ (E^0\setminus H, r_E^{-1}(E^0\setminus H),r_E,s_E) \]
is acyclic, that $v\dm H$ for all $v\in E^0\setminus H$, and that $E^0\setminus H\subset E_\textnormal{reg}^0$. 
Assume furthermore for all $v\in E^0\setminus H$ that there is an upper bound for the length of paths $\alpha\in F(H)$ for which $s_E(\alpha)=v$.
 Then $C^*(E)\cong C^*(E(H))$.
\end{theorem}

\begin{corollary} \label{c:porcupine}
Let $E$ be a graph, and let $G$ be a CK-subgraph of $SE$ containing $E$ as a CK-subgraph.
Then $C^*(G)\cong C^*(G(E^0))$.
\end{corollary}
\begin{proof}
The set $E^0$ is a hereditary subset of $G^0$ satisfying the assumptions of Theorem~\ref{t:EH}.
\end{proof}

\begin{lemma} \label{l:unitization}
Let $E$ be a graph and let $H$ be a hereditary subset of $E^0$. Assume that $H$ is finite.  Then $C^*(\widetilde E(H))\cong C^*(E(H))^\dagger$.
\end{lemma}
\begin{proof}
The \csa{} $C^*(E(H))$ is unital if and only if $F(H) = \{ \alpha\in E^* \mid \alpha = e_1\cdots e_n, s_E(e_n)\notin H, r_E(e_n)\in H \}$ is finite.  If $F(H)$ is finite, then the unique vertex $\star$ of $\widetilde E(H)^0\setminus H$ is regular, so $C^*(\widetilde E(H))\cong C^*((\widetilde E(H))(H))$ by Theorem~\ref{t:EH}.  Since $E(H)=(\widetilde E(H))(H)$, it follows that $C^*(E(H))\cong C^*(\widetilde E(H))$.

Assume that $F(H)$ is not finite.  Then $H$ is a hereditary saturated subset of $\widetilde E(H)$ that has no breaking vertices.  Let $I_H$ denote the corresponding ideal and note that it is nonunital.  
By \cite{rt:ideals}, $I_H\cong C^*(E(H))$ as in this setting their construction corresponds with the construction of $E(H)$.
Since $\widetilde E(H)^0\setminus H$ is the singleton $\{\star\}$, the quotient $C^*(\widetilde E(H))/I_H$ is $\C$.  So since $C^*(\widetilde E(H))$ is unital, we conclude that $I_H^\dagger\cong C^*(\widetilde E(H))$.
\end{proof}

\begin{theorem} \label{thm:unitization}
Let $\A$ be a \spu{} \csa{}. Then $\A^\dagger$ is isomorphic to a graph \csa{} if and only if $\A$ is strongly Morita equivalent to a unital graph \csa{}.
\end{theorem}

\begin{proof}
Suppose $\A{}^{\dagger}$ is isomorphic to a graph $C^{*}$-algebra.  If $\A{}$ is unital, then $\A{}^{\dagger} = \A{}$ so $\A{}$ is strongly Morita equivalent to a unital graph $C^{*}$-algebra.
Suppose $\A{}$ is not unital.  Then $\A$ is a \spu{} hereditary \cssa{} of the unital graph \csa{} $\A^\dagger$.
By Corollary~\ref{c:hereditary}, there exists a graph $F$ with finitely many vertices and a graph $G$ with $F\CKsub G\CKsub SF$ and $\A\cong C^*(G)$.  Note that $C^*(G)\otimes\BK\cong C^*(F)\otimes\BK$.
Hence $\A$ is strongly Morita equivalent to the unital graph \csa{} $C^*(F)$.

Suppose that $\A{}$ is \spu{} and
strongly Morita equivalent to a unital graph $C^{*}$-algebra. 
If $\A{}$ is unital, then $\A=\A^\dagger$ and by Theorem~\ref{t:main}, $\A$ is isomorphic to a graph $C^{*}$-algebra.  Suppose $\A{}$ is not unital.  Let $E$ be a graph with finitely many vertices such that $\A{}$ is strongly Morita equivalent to $C^{*} (E)$.  By Theorem~\ref{t:main}, there exists a graph $F$ with finitely many vertices and a graph $G$ with $F\CKsub G\CKsub SF$ and $\A\cong C^*(G)$.
By Corollary~\ref{c:porcupine}, $C^*(G)\cong C^*(G(F^0))$.
So by Lemma~\ref{l:unitization}, $\A^\dagger\cong C^*(\widetilde G(F^0))$.
\end{proof}

\begin{example} \label{e:unitization}
Theorem~\ref{thm:unitization} fails in both directions when one omits the assumption of \spu{}ity.
For one direction, we note that the \csa{} $C_0(\mathbb R)$ is not stably isomorphic to a graph \csa{} while its minimal unitization $C(S^1)$ is isomorphic to a graph \csa{}.
For the other direction, consider the \csa{} $\A$ defined in Example~\ref{e:spu}.  Its minimal unitization $\A^\dagger$ is not isomorphic to a graph \csa{}, but $\A$ is stably isomphic to the graph \csa{} $C(S^1)$.

To see that $\A^\dagger$ is not isomorphic to a graph \csa{}, note that the primitive ideal space is homeomorphic to $(\mathbb T \times \{1,2\})/\sim$ with $(z,1) \sim (z,2)$ whenever $z \neq 1$.
It can be seen from Theorem~3.4 of~\cite{hs:prim} that no graph \csa{} has such a primitive ideal space. 
(The reader is referred to Theorem~1 of~\cite{gabe} as well, as the statement of the result in~\cite{hs:prim} is not correct.)
\end{example}

As an application of Theorem~\ref{thm:unitization}, we provide the following corollary on semiprojectivity. But first we introduce some terminology.

An important type of ideals in graph $C^\ast$-algebras are the gauge-invariant ideals. However, when we know that a $C^\ast$-algebra is isomorphic to a graph $C^\ast$-algebra, but we do not know the graph, this of course means that we do not have an induced gauge action. So how does one determine the gauge-invariant ideals?

Recall that the ideal lattice in a $C^\ast$-algebra $\A$ is canonically isomomorphic to the lattice of open subsets of the primitive ideal space of $\A$. An ideal in $\A$ is called \emph{compact} if it corresponds to a compact, open subset of the primitive ideal space of $\A$. Note that compact sets need not be closed. For graph $C^\ast$-algebras, it is not hard to see that the compact ideals are exactly the ideals generated by finitely many projections. By Corollary~\ref{c:non-stable-kthy} and by Theorem~3.6 of~\cite{bhrs}, the gauge-invariant ideals in a graph $C^\ast$-algebra are exactly those ideals which are generated by projections. Thus, the gauge-invariant ideals are exactly the ideals which can be formed as a union of the compact ideals. Note that this latter condition can be determined from the primitive ideal space of the $C^\ast$-algebra alone, and does not require that you know any other structure of the $C^\ast$-algebra.

Since all gauge-invariant ideals in a unital graph \csa{} are generated by finitely many projections, the compact ideals of a unital graph \csa{} are exactly the gauge-invariant ideals.  Hence, in the case of a $C^\ast$-algebra $\A$ which is strongly Morita equivalent to a unital graph $C^\ast$-algebra, the compact ideals in $\A$ are exactly the ideals which correspond to the gauge-invariant ideals via the strong Morita equivalence.

\begin{corollary}\label{c:semiproj}
Let $\A$ be a \spu{} \csa{} and assume that $\A$ is strongly Morita equivalent to a unital graph \csa{}.  Then $\A$ is semiprojective if and only if for any compact ideals $\mathfrak I \subseteq \mathfrak J$ in the minimal unitization $\A^\dagger$, $\mathfrak J/\mathfrak I$ is not strong Morita equivalent to $\BK^\dagger$ or $(C(S^1)\otimes\BK)^\dagger$.
\end{corollary}

\begin{proof}
It is well-known that $\A$ is semiprojective if and only if $\A^\dagger$ is semiprojective. By Theorem \ref{thm:unitization}, $\A^\dagger$ is isomorphic to a unital graph $C^\ast$-algebra, so by \cite[Theorem 1.1]{ek:semiprojectivity} it follows that $\A^\dagger$ is semiprojective if and only if for any gauge-invariant ideals $\mathfrak I \subseteq \mathfrak J$ in $\A^\dagger$, the subquotient $\mathfrak J/\mathfrak I$ is not strong Morita equivalent to $\BK^\dagger$ or $(C(S^1)\otimes \mathbb K)^\dagger$. As noted above, the gauge-invariant ideals in a unital graph $C^\ast$-algebra are exactly the compact ideals.
\end{proof}

\section{Acknowledgements}
The authors thank the Department of Mathematics at the University of Hawai'i at Hilo and the Department of Mathematical Sciences at the University of Copenhagen for their hospitality and for providing an excellent working environment.  The authors are grateful to Institut Mittag-Leffler for hosting them during the research program \emph{Classification of operator algebras: complexity, rigidity, and dynamics} in Spring~2016.  This research was supported by the Danish National Research Foundation through the Centre for Symmetry and Deformation (DNRF92) at University of Copenhagen, and a grant from the Simons Foundation (\#279369) to Efren Ruiz.

\newcommand{\etalchar}[1]{$^{#1}$}

\end{document}